\documentclass[11pt]{article}
\usepackage{amsfonts, latexsym, amssymb,amsmath,amsthm}
\usepackage[dvips]{graphicx,color}
\usepackage{caption}
\usepackage{relsize}
\usepackage[mathscr]{eucal}
\usepackage{stmaryrd}
\newtheorem{conjecture}{Conjecture}[section]

\newtheorem{proposition}{Proposition}[section]
\newtheorem{lemma}{Lemma}[section]
\theoremstyle{definition}
\newtheorem{remark}{Remark}[section]

\newcommand{\me}{\mathrm{e}}
\newcommand{\mi}{\mathrm{i}}

\begin{document}
\title{One-Parameter Meromorphic Solution of the Degenerate Third Painlev\'{e} Equation
with Formal Monodromy Parameter $a=\pm\mi/2$ Vanishing at the Origin}
\author{A.~V.~Kitaev\thanks{\texttt{E-mail: kitaev@pdmi.ras.ru}} \, and \,
A.~Vartanian\\
Steklov Mathematical Institute, Fontanka 27, St. Petersburg 191023, Russia}
\date{May 26, 2023}
\maketitle
\begin{abstract}
\noindent
We prove that there exists a one-parameter meromorphic solution $u(\tau)$ vanishing at $\tau=0$ of the degenerate
third Painlev\'e equation,
\begin{equation*}
u^{\prime \prime}(\tau) \! = \! \frac{(u^{\prime}(\tau))^{2}}{u(\tau)} \! - \! \frac{u^{\prime}(\tau)}{\tau}
\! + \! \frac{1}{\tau} \! \left(-8 \varepsilon (u(\tau))^{2} \! + \! 2ab \right) \! + \! \frac{b^{2}}{u(\tau)},\qquad
\varepsilon=\pm1,\quad\varepsilon b>0,
\end{equation*}
for formal monodromy parameter $a=\pm\mi/2$. We study number-theoretic properties of the coefficients of the
Taylor-series expansion of $u(\tau)$ at $\tau=0$ and its asymptotic behaviour as $\tau\to+\infty$. These asymptotics
are visualized for generic initial data.

\vspace{0.30cm}


34M55, 34M56

\vspace{0.23cm}

\textbf{Abbreviated Title.} Meromorphic Solution of the Degenerate Third Painlev\'e

Equation

\vspace{0.23cm}


polynomial
\end{abstract}
\clearpage
\section{Introduction} \label{sec:Introduction}
The degenerate third Painlev\'{e} equation,
\begin{equation} \label{eq:dp3}
u^{\prime\prime}(\tau)=\frac{(u^{\prime}(\tau))^{2}}{u(\tau)}-
\frac{u^{\prime}(\tau)}{\tau}+\frac{1}{\tau} \! \left(-8\varepsilon(u(\tau))^2+2ab\right)+\frac{b^2}{u(\tau)},
\end{equation}
where the prime denotes differentiation with respect to $\tau$, and $\varepsilon,a,b\in\mathbb{C}$ are parameters,
is a special case of the canonical form \cite{Ince} of the third Painlev\'e equation.
The parameter $a\in\mathbb{C}$ is called the formal monodromy, while the parameters $\varepsilon,b\in\mathbb{C}\setminus\{0\}$
can be eliminated via a scaling transformation. Recall that if $\varepsilon b=0$, equation~\eqref{eq:dp3} is solvable in terms
of elementary functions. In problems related with the study of monodromy data and asymptotics, we assume that
$\varepsilon b>0$; otherwise, they are non-vanishing complex numbers.

In the paper \cite{KitSIGMA2019}, we proved that, for all $a\neq0$, there exists a meromorphic solution of equation~\eqref{eq:dp3}
vanishing at $\tau=0$; moreover, this solution is unique, and appears to be an odd function of $\tau$ provided that the following
condition is not valid:
\begin{equation}\label{eq:a-id2+in}
a=\mi/2+\mi n,\qquad
n\in\mathbb{Z}.
\end{equation}
For $a$ satisfying the condition~\eqref{eq:a-id2+in}, there also exists a unique meromorphic solution vanishing at the origin
if, additionally, we require that the solution is an odd function of $\tau$. If, however, we ``break'' the symmetry assumption,
then, for $a$ satisfying the condition~\eqref{eq:a-id2+in}, there is a one-parameter family of meromorphic solutions
vanishing at the origin. Hence, breaking of the symmetry leads to a ``quantization'' of $a$.

In this paper, we study this family of solutions for $a=\pm\mi/2$. The parameter of this family,
denoted by $c_1$, is hidden in the $\tau^2$ term of the Taylor-series expansion of the solution at $\tau=0$, and, thus, breaks
the odd symmetry. The case $c_1=0$, which corresponds to the odd solutions, was studied in \cite{KitSIGMA2019}, among other
symmetric solutions for generic values of $a$; in particular, the monodromy data of the odd solutions are
calculated by appealing to their symmetry property. We recall that the knowledge of the monodromy data allows one to find
large-$\tau$ asymptotics of $u(\tau)$. The standard methodology developed for calculating the monodromy data of general
solutions, that is, the use of Theorems 3.4 and 3.5 from \cite{KitVar2004}, is not directly applicable in this situation
because two parameters specifying the general solution $u(\tau)$ are contained in its leading term while the parameter $c_1$
is located in the next-to-the-leading term. In this work (see Section~\ref{sec:monodromy} below), we show how to cope with
this problem without a detailed study of the corresponding monodromy problem.

It was B. I. Suleimanov who drew our attention to the study of meromorphic solutions of equation~\eqref{eq:dp3}.
In \cite{Suleimanov2017}, he discussed a self-focusing phenomenon for solutions of the equations of nonlinear geometric
optics in the presence of small dispersion in one spatial dimension. In fact, his study concerns a special isomonodromy
solution of the focusing Nonlinear Schr\"odinger Equation (NLSE), which is now accepted as a universal model for the
description of various processes in nonlinear physics. Suleimanov showed that, in the neighbourhood of a focusing point,
the leading-order time-behaviour of the wave can be described by the odd solution of equation~\eqref{eq:dp3} for $a=\mi/2$,
i.e., for the case $c_1=0$.
It seems that this special transcendental solution appeared in the literature for the first time in \cite{Suleimanov2017};
therefore, we suggest calling it ``Suleimanov's solution''. The arguments used by Suleimanov are based on the smoothness of
the solution of the NLSE describing the wave in the neighbourhood of the focusing point. We know, however, that in many
transition phenomena studied in physics, the behaviour of the functions describing these phenomena are not smooth, like, say,
for phase transitions in thermodynamics.
Suleimanov's approach shows that, in the context of the NLSE, there appears, quite naturally, a non-canonical form of
equation~\eqref{eq:dp3}:
\begin{equation}\label{eq:dp3Garnier}
\xi^{\prime\prime}_{tt}=\frac{(\xi^{\prime}_t)^2}{\xi}-\frac{\xi^{\prime}}{t}-\frac{\hat{a}^3\xi^2}{4t^2}+\frac{2\mi}{t}+\frac{16}{\xi},
\end{equation}
where $\xi=\xi(t)$, $\hat{a}$ is a parameter, and $t$ has the same meaning as $t$ in the NLSE
(see \cite{Suleimanov2017}, equation (25)). Equation~\eqref{eq:dp3Garnier} was already considered by Garnier~\cite{Garnier1912};
therefore, we call it the Garnier form of equation~\eqref{eq:dp3}. A relation between solutions of equations~\eqref{eq:dp3}
and \eqref{eq:dp3Garnier} reads:
\begin{equation}\label{eq:xi-u}
\xi(t)=\frac{8\tau}{\hat{a}^3}u(\tau),\qquad
t=\tau^2,\qquad
\varepsilon=1.
\end{equation}
Taking into account a relation between the function $\xi(t)$ and solutions of the NLSE (see \cite{Suleimanov2017}, item 8),
we see that if $u(\tau)$ is a holomorphic function vanishing at the origin and $c_1\neq0$, then the corresponding solution of
the NLSE is not smooth at $t=0$. It would be interesting to see whether or not the more general solutions considered in this
paper may be ascribed a physical interpretation within the framework of the paper \cite{Suleimanov2017}.

The main result of Section~\ref{sec:Taylor} is the existence and uniqueness, for a given $c_1\in\mathbb{C}$, of the solution
of equation~\eqref{eq:dp3} that is holomorphic at the origin. The Painlev\'e property implies that this solution is
meromorphic in $\mathbb{C}$. The Taylor-series coefficients of the expansion at the origin are the polynomials $c_m(c_1)$
with rational coefficients.

In Section~\ref{sec:generating-A}, we construct a so-called Super Generating Function, which is a formal expansion for large
values of $c_1$ whose coefficients are generating functions for the coefficients of the polynomials $c_m$. This formal
expansion can be treated as a special \emph{double asymptotics} of the meromorphic solution as $c_1\to\infty$ with the
condition $xc_1^{1/3}\leqslant\mathcal{O}(1)$; we do not, however, develop this aspect here (a similar construction is given
in \cite{KitSIGMA2019}). The first six generating functions are calculated explicitly. The explicit expressions obtained for
the first three generating functions prove Proposition~\ref{prop:degree-c-m} of Section~\ref{sec:Taylor}.

In  Section~\ref{sec:fence}, we study some number-theoretic properties of the coefficients of the polynomials $c_m(c_1)$,
more precisely, the divisibility properties of the \emph{content} of these polynomials. This section has a purely experimental
character, and is based on the calculation of the first $800$ polynomials $c_m(c_1)$. Our calculations show that the coefficients
of these polynomials are positive rational numbers. The content of $c_m(c_1)$ is a natural power of $2$. The results of the
study of the 2-adic valuation of the content of $c_m(c_1)$ is formulated in terms of two plots, which are made separately,
one for polynomials with odd subscripts, and another for polynomials with even subscripts. As a consequence of the properties
of these plots, we call them the odd and even quasi-periodic fences, respectively. The analysis of these fences is undertaken
with the help of the On-Line Encyclopedia of Integer Sequences (OEIS),\footnote{\label{foot:OEIS} https://oeis.org}
which helped us to formulate a hypothesis with explicit formulae for the 2-adic valuation of the content of the
polynomials $c_m(c_1)$.

In Section~\ref{sec:monodromy}, we calculate the monodromy data of the vanishing meromorphic solutions
for $a=\pm\mi/2$.

In Appendix~\ref{app:pictures}, we combine the results for the monodromy data (cf. Section~\ref{sec:monodromy})
with the large-$\tau$ asymptotics for solutions of equation~\eqref{eq:dp3} \cite{KitVar2004,KitVar2010} in order
to visualize the generic behaviour of the meromorphic solutions studied in this paper.
\section{Expansion as $\tau \to 0$} \label{sec:Taylor}
In the study of the small-$\tau$ expansion of solutions to equation~\eqref{eq:dp3} for $a=\pm\mi/2$, it is convenient to
adopt the following normalization for the coefficients and the solution:
\begin{equation} \label{eqs:y-x-b=a}
b=a=\kappa\,\mi/2,\quad\quad
\kappa=\pm1,\quad\quad
\tau=x,\quad\quad
u(\tau)=y(x);
\end{equation}
then, $y(x)$ solves the ODE
\begin{equation} \label{eq:ydp3}
y^{\prime \prime}(x)=\frac{(y^{\prime}(x))^{2}}{y(x)}-\frac{y^{\prime}(x)}{x}
-\frac{1}{x}\!\left(8\varepsilon(y(x))^{2}+\frac12\right)-\frac{1}{4y(x)},
\end{equation}
where, now, the prime denotes differentiation with respect to $x$.
\begin{remark} \label{rem:dp3-ydp3}
If $y(x)$ is a solution of equation~\eqref{eq:ydp3},
then function $u(\tau)$ defined as
\begin{equation}\label{eq:changevar}
x=\alpha\,\tau,\quad\quad
u(\tau)=\alpha\,y(x),\quad\quad
\alpha=\sqrt{-2\kappa\,\varepsilon b\,\mi},
\end{equation}
solves equation~\eqref{eq:dp3} for $a=\kappa\,\mi/2$.
\hfill $\blacksquare$\end{remark}
Consider the following formal expansion for the function $y(x)$ as $x\to0$:
\begin{equation} \label{eq:taylor-expansion}
y(x) \! = \! -\frac{x}{2} \! \left(1 \! + \! \sum_{m=1}^{\infty}c_{m}x^m \right),
\end{equation}
where the coefficients $c_m$, $m=\mathbb{N}$, are independent of $x$.
Substituting the expansion~\eqref{eq:taylor-expansion} into equation~\eqref{eq:ydp3},
one arrives at the following recurrence relation for $c_{m}$, $m\geqslant2$:
\begin{equation} \label{eq:recurrence}
(m^{2} \! - \! 1)c_{m} \! = \! \sum_{p=0}^{m-2}(p \! + \! 2)(m \! - \! 2(p \! + \! 1))c_{p+1}
c_{m-p-1} \! + \! 4 \sum_{p=0}^{m-2} \sum_{p_{1}=0}^{p}c_{m-p-2}c_{p_{1}}c_{p-p_{1}},
\end{equation}
where $c_0=1$ and $c_1\in\mathbb{C}$ is a parameter.
The first few coefficients are
\begin{equation} \label{eqs:c1-9}
\begin{gathered}
c_{2} \! = \! \frac{4}{3}, \, \qquad \, c_{3} \! = \! \frac{4}{3}c_{1}, \, \qquad \, c_{4} \! = \!
\frac{4}{9}c_{1}^{2} \! + \! \frac{16}{15}, \, \qquad \, c_{5} \! = \! \frac{206}{135}c_{1}, \,
\qquad \, c_{6} \! = \! \frac{512}{675}c_{1}^{2} \! + \! \frac{256}{315}, \\
c_{7} \! = \! \frac{4}{27}c_{1}^{3} \! + \! \frac{1336}{945}c_{1}, \, \qquad \, c_{8} \! = \!
\frac{10768}{11025}c_{1}^{2} \! + \! \frac{4864}{8505}, \, \qquad \, c_{9} \! = \!
\frac{1936}{6075}c_{1}^{3} \! + \! \frac{253774}{212625}c_{1}.
\end{gathered}
\end{equation}
As can be seen {}from equations~\eqref{eqs:c1-9}, the coefficients $c_{m}=c_{m}(c_{1})$
are polynomials of $c_{1}$ over $\mathbb{Q}$. It is more complicated, though, to prove the
following statements.
\begin{proposition} \label{prop:parity}
The parity of the natural number $m$ coincides with the parity of the polynomial $c_{m}(c_{1})$
considered as a  function of $c_{1}$.
\end{proposition}
\begin{proof}
The proof is done by mathematical induction. The base of the induction argument follows {}from
equations~\eqref{eqs:c1-9}, and the induction step is a straightforward consequence of
the recurrence relation~\eqref{eq:recurrence}.
\end{proof}
\begin{proposition} \label{prop:degree-c-m}
For $k \! \in \! \mathbb{Z}_{+} \! := \! \lbrace 0 \rbrace \cup \mathbb{N}$,
\begin{equation} \label{eq:degree-c-m}
\deg (c_{m}(c_{1})) \! = \!
\left\{
\begin{aligned}
&k, \quad m \! = \! 3k \! \quad \! \text{and} \! \quad \! m \! = \! 3k+2, \\
&k \! + \! 1, \quad m \! =\! 3k+1.
\end{aligned}
\right.
\end{equation}
\end{proposition}
\begin{proof}
This proposition is an immediate corollary of the following lemma and Proposition~\ref{prop:senior-coeffs}.
\end{proof}
\begin{lemma}\label{lem:cm-inequality}
For $k\in\mathbb{Z}_{+}$,
\begin{equation} \label{ineq:degree-c-m}
\deg (c_{m}(c_{1})) \! \leqslant\!
\left\{
\begin{aligned}
&k, \quad m \! = \! 3k \! \quad \! \text{and} \! \quad \! m \! = \! 3k+2, \\
&k \! + \! 1, \quad m \! =\! 3k+1.
\end{aligned}
\right.
\end{equation}
\end{lemma}
\begin{proof}The proof uses mathematical induction.

The base of the induction argument is proven upon inspection of equations~\eqref{eqs:c1-9}:
$\deg c_0=\deg c_2=0$, $\deg c_1=1$, $\deg c_3=\deg c_5=1$, and $\deg c_4=2$,
because $0=3\times0$ and $2=3\times0+2$, $1=3\times0+1$, $3=3\times1$ and $5=3\times1+2$, and
$4=3\times1+1$, respectively. Thus, the base of induction is verified for the equality sign in \eqref{ineq:degree-c-m},
as it should be according to Proposition~\ref{prop:degree-c-m}.

The inductive step is a straightforward, but tedious, exercise, since it requires consideration
of the three cases $m=3k$, $m=3k+1$, and $m=3k+2$.
Let's consider, for example, the case $m=3k$, assuming that the statement of the lemma is valid for
all $m_1<m$. We calculate the degree of the $p$th polynomial in the first sum of \eqref{eq:recurrence}.
We have to consider the following three subcases:

$(1)\; p=3l \Rightarrow p+1=3l+1 \Rightarrow \deg\,c_{p+1}\leqslant l+1 \Rightarrow m-p-1=3(k-l-1)+2
\Rightarrow \deg\,c_{m-p-1}\leqslant k-l-1 \Rightarrow \deg\,c_{p+1}c_{m-p-1}\leqslant k$;

$(2)\; p=3l+1 \Rightarrow p+1=3l+2 \Rightarrow \deg\,c_{p+1}\leqslant l \Rightarrow m-p-1=3(k-l-1)+1
\Rightarrow \deg\,c_{m-p-1}\leqslant k-l \Rightarrow \deg\,c_{p+1}c_{m-p-1}\leqslant k$;

$(3)\; p=3l+2 \Rightarrow p+1=3l+3 \Rightarrow \deg\,c_{p+1}\leqslant l+1 \Rightarrow m-p-1=3(k-l-1)
\Rightarrow \deg\,c_{m-p-1}\leqslant k-l-1 \Rightarrow \deg\,c_{p+1}c_{m-p-1}\leqslant k$.

We now proceed to calculate the degree of the polynomial in the double sum in \eqref{eq:recurrence} labeled by $p$ and $p_1$.
Let $p=3l$; then, irrespective of the value of the summation index $p_1$, $\deg\,c_{p_1}c_{p-p_1}\leqslant l$: the proof
is the same as the one for the first sum considered above, but with changes $m\to p$, $p+1\to p_1$, and $k\to l$. Since,
according to the induction assumption, $\deg\,c_{m-p-2}\leqslant k-l$, we get that
$\deg\,c_{m-p-2}c_{p_1}c_{p-p_1}\leqslant k$. Similarly, one proves that
$\deg\,c_{m-p-2}c_{p_1}c_{p-p_1}\leqslant k$ for $p=3l+1$ and $p=3l+2$.

Thus, it is proven that $c_m$ is the sum of polynomials of the degree $\leqslant k$, which means that $\deg\,c_m\leqslant k$.
The remaining cases in \eqref{ineq:degree-c-m} do not present any additional difficulties, and can be verified analogously.
\end{proof}
\begin{remark}\label{rem:cm-lemma}
There may be a question as to why Proposition~\ref{prop:degree-c-m} can not be verified in the same way as
Lemma~\ref{lem:cm-inequality}.
By inspection, we see that the leading coefficients of the polynomials $c_m(c_1)$ are positive, thus one may suspect that
a slightly extended induction hypothesis, which subsumes this fact, may be proved.
The problem involved, however, is that the first sum, after combining together like terms, has an overall minus sign
(see equation~\eqref{eq:S1-converted} below);
therefore, without \emph{a priori} knowledge of the behaviour of the leading coefficients of the polynomials
$c_m(c_1)$, such an extended induction hypothesis is difficult to formulate.
Proposition~\ref{prop:senior-coeffs} (see Section~\ref{sec:generating-A} below) resolves this problem.
\hfill $\blacksquare$\end{remark}
\begin{proposition} \label{prop:convergence}
The series~\eqref{eq:taylor-expansion} converges for all $c_{1} \! \in \! \mathbb{C}$, that is,
it represents the Taylor-series expansion of a solution of equation~\eqref{eq:ydp3}
at $x=0$.\footnote{Via the substitutions~\eqref{eqs:y-x-b=a}, it also represents the Taylor-series expansion
of a solution of equation~\eqref{eq:dp3} at $\tau=0$.}
\end{proposition}
\begin{proof}
In contrast to the more precise evaluation for the degrees of the polynomials $c_m$ as in Lemma~\ref{lem:cm-inequality},
it is easy to prove by induction that $\deg (c_{m}(c_{1}))<m$; therefore, it is evident that one can choose real
numbers $\alpha\geqslant1$ and $C \! > \! 1$ such that the estimate
\begin{equation} \label{ineq:c-m estimate}
\lvert c_{m} \rvert \! < \! \frac{\alpha C^{m}}{(m \! + \! 1)^2}
\end{equation}
holds for all $m \! = \! 0,1,\ldots,M$ for any \emph{a priori} fixed $M\in\mathbb{Z}_{+}$.
It will be shown that there exists a large enough $M$ such that if the
inequality~\eqref{ineq:c-m estimate} is true for all $m<M$, then it is valid for all $m\geqslant M$.

The first sum on the right-hand side of equation~\eqref{eq:recurrence} can be rewritten as
\begin{equation}\label{eq:S1-converted}
S_{1}:=\sum_{p=0}^{m-2}(p+2)(m-2(p+1))c_{p+1}c_{m-p-1}=-\sum_{p=0}^{\lfloor \frac{m-2}2\rfloor}(m-2p-2)^2c_{p+1}c_{m-p-1},
\end{equation}
where $\lfloor*\rfloor$ is the floor of the rational number $*$. (The upper limit in the last sum can also be written as
$\lfloor \frac{m-3}2\rfloor$.) Using the induction conjecture, we evaluate this sum as follows:
\begin{equation}\label{ineq:S1-estimate}
\begin{aligned}
|S_1|&\leqslant\sum_{p=0}^{\lfloor \frac{m-2}2\rfloor}\frac{(m-2p-2)^2}{(p+2)^2(m-p)^2}\alpha^2C^m
=\left(\sum_{p=0}^{m-2}\frac1{(p+2)^2}-\frac{2}{m+2}\sum_{p=0}^{m-2}\frac1{(p+2)}\right)\alpha^2C^m\\
&<\left(\frac{\pi^2}{6}-1-\frac{2(\ln m+\gamma-1)}{m+2}\right)\alpha^2C^m,
\end{aligned}
\end{equation}
where  $\gamma=0.57721\ldots$ is the Euler constant.

We now evaluate the double sum in equation~\eqref{eq:recurrence}:
\begin{equation*}
S_{2}:=4\sum_{p=0}^{m-2} \sum_{p_{1}=0}^{p}c_{m-p-2}c_{p_{1}}c_{p-p_{1}}.
\end{equation*}
It, too, can be bounded in an analogous manner, namely, the induction conjecture implies
\begin{equation} \label{ineq:S2-induction}
\lvert S_{2} \rvert \! < \! \left(\sum_{p=0}^{m-2} \frac{4}{(m \! - \! p \! - \! 1)^2} \,
\sum_{p_{1}=0}^{p} \frac{1}{(p_{1} \! + \! 1)^2(p \! - \! p_{1} \! + \! 1)^2} \right) \!
\alpha^{3}C^{m-2}.
\end{equation}
For the inner sum in \eqref{ineq:S2-induction}, one writes
\begin{equation}\label{ineq:S2-inner}
\!\sum_{p_{1}=0}^{p} \frac{1}{(p_{1} \! + \! 1)^2(p \! - \! p_{1} \! + \! 1)^2}
=\frac{2}{(p+2)^2}\sum_{p_{1}=0}^{p}\!\left(\frac{1}{(p_1+1)^2}+\frac{2}{(p+2)(p_1+1)}\right)\!
<\frac{\frac{\pi^2}{3}+2}{(p+2)^2}.
\end{equation}
Substituting the estimate~\eqref{ineq:S2-inner} into the inequality~\eqref{ineq:S2-induction}, we find that
\begin{equation}\label{ineq:S2-estimate}
|S_2|<\frac{4(\frac{\pi^2}{3}+2)\alpha^3C^{m-2}}{(m+1)^2}\sum_{p=0}^{m-2}\left(\frac{1}{p+2}+\frac{1}{m-p-1}\right)^2
<\frac{4(\frac{\pi^2}{3}+2)(\frac{\pi^2}{3}+1)\alpha^3C^{m-2}}{(m+1)^2}.
\end{equation}

Combining the estimate~\eqref{ineq:S2-estimate} with a simplified version (without the $\mathcal{O}\big(\ln\,m/m\big)$
correction term) of \eqref{ineq:S1-estimate}, we get the following estimate for $c_m$:
\begin{equation}\label{ineq:c-m}
|c_m|<\frac{m+1}{m-1}\alpha\left(\frac{\pi^2}{6}-1
+\frac{4(\frac{\pi^2}{3}+2)(\frac{\pi^2}{3}+1)\alpha}{C^2(m+1)^2}\right)\frac{\alpha C^m}{(m+1)^2}.
\end{equation}
Therefore, in order to complete the inductive step, we have to show that
\begin{equation}\label{eq:Rm}
R_m:=\frac{m+1}{m-1}\alpha\left(\frac{\pi^2}{6}-1+\frac{4(\frac{\pi^2}{3}+2)(\frac{\pi^2}{3}+1)\alpha}{C^2(m+1)^2}\right)
\leqslant1.
\end{equation}
The function $R_m$ is a monotonically increasing function of $\alpha\geqslant1$ and a monotonically decreasing function of $m$
and $C$. If one assumes that $\alpha\in[1,1.04]$ and $C^2\geqslant12$, then it is enough to verify that $R_m\leqslant1$ for
the boundary values $\alpha=1.04$ and $C^2=12$. Actually, with the help of \textsc{Maple}, one finds that
$R_8=0.9922344425\ldots$, while $R_7=1.064756992\ldots$; thus, $R_m<1$ for $m>7$, and $R_m>1$ for $m=2,\ldots,7$. We see that
the induction procedure can be initiated from $M=7$. Since $\alpha$ is bounded, it is then obvious that, after---perhaps---an
increase of $C$, one can reach validity for the base of induction (cf. \eqref{ineq:c-m estimate}) for any \emph{a priori}
given $M$.
\end{proof}
\begin{remark}\label{rem:minC^2increase}
Below, we find a minimum
increase of $C$, if any, such that the condition~\eqref{ineq:c-m estimate} is valid for the first 8 members of the sequence
$c_m$, $m=0,\ldots,7$:
\begin{equation}\label{ineq:c0-c2}
c_0=1\leqslant\alpha,\quad
|c_1|\leqslant\frac{\alpha C}{4},\quad
c_2=\frac43=\frac{12}{9}\leqslant\frac{\alpha C^2}{9}.
\end{equation}
The second and third inequalities in \eqref{ineq:c0-c2} suggest the condition $C\geqslant\max\{\sqrt{12/\alpha},4|c_1|/\alpha\}$.
This slight improvement of the minimal value for $C$ (cf. the proof of Proposition~\ref{prop:convergence}) does not change
the values of $m$ for which we have to verify the base of the induction procedure because $R_8=0.9974290608\ldots$ and
$R_7=1.071574929\ldots$.
This is the final condition on the parameter $C$ which follows from our proof. The estimates for the remaining coefficients
$c_3,\ldots, c_7$ can be obtained by following the scheme demonstrated below for estimating the coefficients
$c_3$ and $c_4$. In these estimates, we rely upon the inequalities \eqref{ineq:c0-c2} and assume that
$\alpha\leqslant 7/5=1.4$:
\begin{gather}\label{ineq:c3c4}
c_3=\frac43c_1\leqslant\frac{\alpha C^2}{9}\cdot\frac{\alpha C}{4}=\frac{4\alpha}{9}\cdot\frac{\alpha C^3}{4^2}
\leqslant\frac{28}{45}\cdot\frac{\alpha C^3}{4^2}<\frac{\alpha C^3}{4^2},\;
c_4=\frac13\cdot c_1\cdot\frac43c_1+\frac95\left(\frac43\right)^2\nonumber\\
\leqslant\frac13\cdot\frac{\alpha C}{4}\cdot\frac{\alpha C^3}{4^2}+\frac95\left(\frac{\alpha C^2}{9}\right)^2
=5\alpha\left(\frac19+\frac5{3\cdot4^3}\right)\frac{\alpha C^4}{5^2}\leqslant\frac{7\cdot79}{576}\frac{\alpha C^4}{5^2}
<\frac{\alpha C^4}{5^2}.\nonumber
\end{gather}
The reader who has verified the induction hypothesis for the remaining coefficients $c_5$, $c_6$, and $c_7$ for $\alpha=1.4$
may be interested as to whether or not a further increase of $\alpha$ is possible? (Note that $\alpha=1.4$ represents the
increase in the value of $\alpha=1.04$ given in the proof of Proposition~\ref{prop:convergence}.)
This increase constitutes a better estimate for the
radius of convergence of the series~\eqref{eq:taylor-expansion}; for example, to complete the above proof for $\alpha=1.4$,
one has to verify the validity of the induction hypothesis for the coefficients $c_m$, $m\leqslant27$. In fact,
the logarithmic correction (cf. \eqref{ineq:S1-estimate}), which we omitted in order to simplify the proof, helps considerably
in validating the induction hypothesis; for $\alpha=1.4$, say, by keeping this correction in $R_m$, it is sufficient to check
the induction hypothesis only for $c_0,\ldots, c_7$, as in the case for $\alpha=1.04$. The inclusion of the logarithmic
correction in $R_m$, however, requires
a more careful study of the dependence of $R_m$ on $m$; in particular, one has to verify that, initially, it is decreasing
to some minimal value, and then monotonically increasing up to $\alpha\big(\pi^2/6-1\big)$. Clearly, the theoretical limit for
the value of $\alpha$ in the proof of Proposition~\ref{prop:convergence} is $1/\big(\pi^2/6-1\big)=1.550546094\ldots$.
\hfill$\blacksquare$\end{remark}
\section{Generating Functions} \label{sec:generating-A}
The coefficients of the polynomials $c_{m}(c_{1})$ can be calculated with the help of
so-called Super Generating Functions (SGFs) \cite{KitSIGMA2019}. Since the polynomials $c_m(c_1)$, $m\geqslant3$, possess only
one singular point at $c_1=\infty$, we can define only one SGF. The idea behind this definition is to visualize the parameter
$c_1$, which is hidden as a parameter in the Taylor-series expansion characterizing the solution introduced in
Section~\ref{sec:Taylor}, as a parameter in the coefficients of an equation related to \eqref{eq:ydp3}. We achieve this
visualization by using a scaling transformation for equation~\eqref{eq:ydp3}. Bearing in mind Lemma~\ref{lem:cm-inequality},
this transformation can be guessed; in this work, the SGF, $w(z,c_{1})$, is defined as follows:
\begin{equation}\label{eq:rescaling}
y(x)=-\frac{c_{1}^{1/3}}{2}zw(z,c_1),\qquad
x=\frac{z}{c_{1}^{1/3}},
\end{equation}
where $y(x)$ is the solution studied in Section~\ref{sec:Taylor}, and
\begin{equation}\label{eq:w}
\left(z \frac{w^{\prime}}{w} \right)^{\prime} \! = \! 4zw \! + \! \frac{1}{c_{1}^{2/3}zw} \! - \!
\frac{1}{c_{1}^{4/3}zw^{2}},
\end{equation}
where $w=w(z,c_1)$, and the prime denotes differentiation with respect to $z$. As soon as equation~\eqref{eq:w} is written,
we find ourselves in a situation similar to that studied in the work~\cite{KitSIGMA2019}.
In order to fulfill the definition of the SGF, it is necessary to specify its asymptotic expansion as $c_{1} \! \to \! \infty$:
\begin{equation} \label{eq:w-Ak}
w \! = \! \sum_{k=0}^\infty \frac{A_{k}(z)}{c_{1}^{2k/3}},
\end{equation}
where the coefficients $A_k(z)$  are assumed to be independent of $c_{1}$. As will be seen below, the coefficients
$A_{k}(z)$ serve as the generating functions for the coefficients of the polynomials $c_{m}(c_{1})$.

For brevity, we often denote $A_k(z)=A_k$. Substituting the expansion~\eqref{eq:w-Ak} into equation~\eqref{eq:w},
one shows that
\begin{align}
\left(z \frac{A_{0}^{\prime}}{A_{0}} \right)^{\prime} &= 4zA_{0}, \label{eq:A0} \\
\left(z \! \left(\frac{A_{1}}{A_{0}} \right)^{\prime} \right)^{\prime} \! - \! 4zA_{1} &=
\frac{1}{zA_{0}}, \label{eq:A1} \\
\left(z \! \left(\frac{A_{2}}{A_{0}} \right)^{\prime} \right)^{\prime} \! - \! 4zA_{2} &=
\left(\frac{z}{2} \! \left(\frac{A_{1}^{2}}{A_{0}^{2}} \right)^{\prime} \right)^{\prime} \!
- \! \frac{(A_{1} \! + \! 1)}{zA_{0}^{2}}, \label{eq:A2} \\
\left(z \! \left( \frac{A_{n}}{A_{0}} \right)^{\prime} \right)^{\prime} \! - \! 4zA_{n} &=
f_{n}(A_{1},\ldots,A_{n-1}), \quad n \! \geqslant \! 3, \label{eq:An}
\end{align}
where
\begin{align}
f_{n}(A_{1},\ldots,A_{n-1}) &:= \sum_{\substack{i_{1}+\cdots+i_{k}+j=n \\ 1 \leqslant j
\leqslant n-1 \\ i_{k} \geqslant 1, \; \; k \geqslant 1}}(-1)^{k+1} \! \left(z \alpha_{i_{1}}
\cdots \alpha_{i_{k}} \alpha_{j}^{\prime} \right)^{\prime} \nonumber \\
+ \, \frac{1}{zA_{0}} \sum_{\substack{i_{1}+\cdots+i_{m}=n-1 \\ i_{k} \geqslant 1}}(-1)^{m}
\alpha_{i_{1}} \cdots \alpha_{i_{m}} &- \frac{1}{zA_{0}^{2}} \sum_{\substack{i_{1}+ \cdots
+ i_{m}=n-2 \\ i_{k} \geqslant 1}}(-1)^{m}(m \! + \! 1) \alpha_{i_{1}} \cdots \alpha_{i_{m}},
\nonumber
\end{align}
with $\alpha_{k} \! = \! A_{k}/A_{0}$, $k \! \in \! \mathbb{N}$; for example, the first two
functions $f_{n} \! := \! f_{n}(A_{1},\ldots,A_{n-1})$, $n \! = \! 3,4$, are given by
\begin{align*}
f_{3} &= \left(z \! \left(\frac{A_{1}A_{2}}{A_{0}^{2}} \right)^{\prime} \right)^{\prime} \! - \!
\left(\frac{z}{3} \! \left(\frac{A_{1}^{3}}{A_{0}^{3}} \right)^{\prime} \right)^{\prime} \! + \!
\frac{1}{zA_{0}^{2}} \! \left(-A_{2} \! + \! \frac{A_{1}^{2} \! + \! 2A_{1}}{A_{0}} \right),\\
f_{4} &= \left(z \! \left(\frac{A_{1}A_{3}}{A_{0}^{2}} \right)^{\prime} \right)^{\prime} \! + \!
\left(\frac{z}{2} \! \left(\frac{A_{2}^{2}}{A_{0}^{2}} \right)^{\prime} \right)^{\prime} \! + \!
\left(\frac{z}{4} \! \left(\frac{A_{1}^{4}}{A_{0}^{4}} \right)^{\prime} \right)^{\prime} \! - \!
\left(z \! \left(\frac{A_{1}^{2}A_{2}}{A_{0}^{3}} \right)^{\prime} \right)^{\prime} \\
&+ \frac{1}{zA_{0}^{2}} \! \left(-A_{3} \! + \! \frac{2A_{2}(A_{1} \! + \! 1)}{A_{0}} \! - \!
\frac{(A_{1}^{3} \! + \! 3A_{1}^{2})}{A_{0}^{2}} \right).
\end{align*}

The immediate goal is to apply the construction of the SGF to the
solution~\eqref{eq:taylor-expansion}. It is straightforward to verify that, after the
rescaling~\eqref{eq:rescaling} of the solution~\eqref{eq:taylor-expansion}, the corresponding
function $w(z)$ has the expansion~\eqref{eq:w-Ak} as $c_{1} \! \to \! \infty$, where the
functions $A_{k}(z)$ are single-valued; in fact, the expansion~\eqref{eq:w-Ak} can formally be
obtained as the re-expansion of the Taylor series~\eqref{eq:taylor-expansion} as $c_{1}\to\infty$.
This construction can be put on non-formal footing as done in the analogous situation in \cite{KitSIGMA2019};
we do not repeat the corresponding arguments here.

One begins with the function $A_{0}(z)$: it can be defined as the single-valued solution of
equation~\eqref{eq:A0} with asymptotics
\begin{equation} \label{eq:A0asympt0}
A_{0}(z) \underset{z \to 0} = z+\mathcal{O}\big(z^{2}\big).
\end{equation}
Equation~\eqref{eq:A0} can be viewed as a solvable case of the third Painlev\'{e} equation;
its general solution reads:\footnote{\label{foot:solspecial} There is also a special solution $A_0(z)=z^{-2}(C_0-\ln\,z)^{-2}/2$
which is not single-valued at $z=0$.}
\begin{equation*}
A_{0}(z) \! = \! -\frac{C_{2}C_{1}^{2}z^{C_{1}}}{2z^{2} \! \left(1 \! + \! C_{2}z^{C_{1}}
\right)^{2}},
\end{equation*}
where $C_{1}$ and $C_{2}$ are arbitrary constants of integration. Applying the
expansion~\eqref{eq:A0asympt0}, one deduces that $C_{1} \! = \! 3$ and $C_{2} \! = \! -2/9$, whence
\begin{equation} \label{eq:A0-rational}
A_{0}(z) \! = \! \frac{z}{\left(1 \! - \! 2z^{3}/9 \right)^{2}}.
\end{equation}
In view of equation~\eqref{eq:A0-rational}, one notes that equations~\eqref{eq:A1}--\eqref{eq:An}
are special inhomogeneous hypergeometric equations; their homogeneous part has the solution
\begin{equation*}
C_{3} \frac{z(2z^{3} \! + \! 9)}{(2z^{3} \! - \! 9)^{3}} \! + \! C_{4} \frac{z(2z^{3} \ln{z} \! + \! 9
\ln{z} \! + \! 12)}{(2z^{3} \! - \! 9)^{3}},
\end{equation*}
where $C_{3}$ and $C_{4}$ are arbitrary constants of integration. Since our objective is to look
for single-valued solutions, one must set $C_{4} \! = \! 0$. The parameter $C_{3}$ should be
chosen in accordance with the expansion of $A_{n}(z)$ as $z \! \to \! 0$; in fact, all of the functions
$A_{k}(z)$ are rational functions of $z$. The functions $A_{1}(z)$ and $A_{2}(z)$ correspond to the
parameter $C_{3} \! = \! 0$:
\begin{equation} \label{eq:A1-A2-rational}
\begin{gathered}
A_{1}(z) \! = \! \frac{(2z^{3} \! + \! 45)(4z^{6} \! - \! 252z^{3} \! - \! 405)}{25(2z^{3} \! - \!
9)^{3}}, \\
A_{2}(z) \! = \! -\frac{8}{16875}z^{5} \! + \! \frac{1108}{91875}z^{2} \! - \! \frac{162z^{2}
(340z^{9} \! - \! 14112z^{6} \! - \! 436509z^{3} \! - \! 1638792)}{30625(2z^{3} \! - \! 9)^{4}}.
\end{gathered}
\end{equation}
The functions $A_{0}(z)$, $A_{1}(z)$, and $A_{2}(z)$ are generating functions for the leading
coefficients of the polynomials $c_{3k+1}$, $c_{3k}$, and $c_{3k+2}$, respectively. Expanding
$A_{0}(z)$, $A_{1}(z)$, and $A_{2}(z)$, respectively, as $z \! \to \! 0$, one shows that
\begin{gather*}
A_{0}(z) \! = \! z \! + \! \frac{4}{9}z^{4} \! + \! \frac{4}{27}z^{7} \! + \! \frac{32}{729}z^{10}
\! + \! \frac{80}{6561}z^{13} \! + \! \frac{64}{19683}z^{16} \! + \! \frac{448}{531441}z^{19}
\! + \! \mathcal{O} \! \left(z^{22} \right), \\
A_{1}(z) \! = \! 1 \! + \! \frac{4}{3}z^{3} \! + \! \frac{512}{675}z^{6} \! + \! \frac{1936}{6075}
z^{9} \! + \! \frac{6272}{54675}z^{12} \! + \! \frac{18496}{492075}z^{15} \! + \!
\frac{2048}{177147}z^{18} \!+ \! \mathcal{O} \! \left(z^{21} \right), \\
A_{2}(z) \! = \! \frac{4}{3}z^{2} \! + \! \frac{206}{135}z^{5} \! + \! \frac{10768}{11025}z^{8}
\! + \! \frac{1174888}{2480625}z^{11} \! + \! \frac{290816}{1488375}z^{14} \! + \!
\frac{14567072}{200930625}z^{17} \!+ \! \mathcal{O} \! \left(z^{20} \right).
\end{gather*}
Proposition~\ref{prop:degree-c-m} allows one to introduce the following notation for the
coefficients $p_{m,n}$ of the polynomials $c_{m}(c_{1})$:
\begin{equation} \label{eq:cm-structure}
c_{m}(c_{1}) \! = \sum_{n=0}^{r_{m}}p_{m,n}c_{1}^{[m/3]+\delta-2n}, \quad r_{m} \! := \!
[([m/3]+\delta)/2],
\end{equation}
where $[\ast]$ denotes the entire part of $\ast$, and
\begin{equation*}
\delta \! = \!
\begin{cases}
0, &\text{$m \! = \! 3k$ \, \text{and} \, $m \! = \! 3k \! + \! 2$,} \\
1, &\text{$m \! = \! 3k \! + \! 1$.}
\end{cases}
\end{equation*}
Using the explicit formulae for the generating functions $A_{0}(z)$, $A_{1}(z)$, and $A_{2}(z)$, one
derives expressions for the leading coefficients $p_{m,0}$.
\begin{proposition} \label{prop:senior-coeffs}
\begin{gather*}
p_{3k+1,0} \! = \! (k \! + \! 1) \! \left(\frac{2}{9} \right)^{k}, \quad k \! \in \! \mathbb{Z}_{+}, \\
p_{0,0} \! = \! 1, \quad p_{3k,0} \! = \! \frac6{25}(3k \! + \! 2)^{2} \! \left(\frac{2}{9} \right)^{k},
\quad k \! \in \! \mathbb{N}, \\
p_{2,0} \! = \! \frac{4}{3}, \quad p_{5,0} \! = \! \frac{206}{135}, \quad p_{3k+2,0} \! = \!
\frac{9}{30625}(196k \! + \! 281)(3k \! + \! 4)^{2} \! \left(\frac{2}{9} \right)^{k}, \; \; k \!
= \! 2,3,\ldots.
\end{gather*}
\end{proposition}

\begin{proof} Substituting equation~\eqref{eq:cm-structure} into the Taylor
series~\eqref{eq:taylor-expansion} and rearranging the summation, one arrives at the formal
expansion~\eqref{eq:w-Ak}, where
\begin{gather*}
A_{0}(z) \! = \! \sum_{k=0}^{\infty}p_{3k+1,0}z^{3k+1}, \qquad A_{1}(z) \! = \!
\sum_{k=0}^{\infty}p_{3k,0}z^{3k}, \qquad A_{2}(z) \! = \! \sum_{k=0}^{\infty}
p_{3k+2,0}z^{3k+2}.
\end{gather*}
Expanding the rational functions~\eqref{eq:A0-rational} and~\eqref{eq:A1-A2-rational} into
Taylor series completes the proof.
\end{proof}
The subsequent triple of functions $A_{3}(z)$, $A_{4}(z)$, and $A_{5}(z)$ are generating functions for
the coefficients $p_{m,1}$.
\begin{proposition} \label{prop:p-m1-coeffs}
\begin{gather*}
p_{4,1} \! = \! \frac{16}{15}, \qquad p_{7,1} \! = \! \frac{1336}{945}, \\
p_{3k+1,1} \! = \! \frac{2}{5^{6}} \! \left(2916k^{2} \! + \! \frac{328779}{49}k \! - \!
\frac{34129}{147} \right) \! (k+1)^{2} \! \left(\frac{2}{9} \right)^{k}, \; \; k \! = \! 3,4,\ldots,
\end{gather*}
\begin{gather*}
p_{6,1} \! = \! \frac{256}{315}, \qquad p_{9,1} \! = \! \frac{253774}{212625}, \qquad
p_{12,1} \! = \! \frac{4788251008}{4862521125}, \\
p_{3k,1} \! = \! \frac{1}{5^{7}} \! \left(\frac{8748}{5}k^{3} \! + \! \frac{223074}{49}k^{2} \!
- \! \frac{281982223}{48020}k \! - \! \frac{15481989}{41503} \right) \! (3k \! + \! 2)^{2} \!
\left(\frac{2}{9} \right)^{k}, \; \; k \! = \! 5,6,\ldots,
\end{gather*}
\begin{align*}
&\qquad p_{8,1} \! = \! \frac{4864}{8505}, \qquad p_{11,1} \! = \! \frac{3958936}{4209975},
\qquad p_{14,1} \! = \! \frac{44744664088576}{51771262417875}, \\
p_{3k+2,1} \! =& \, \frac{3}{5^{10}} \! \left(\frac{34992}{5}k^{4} \!+ \! \frac{10865016}{245}
k^{3} \! + \! \frac{86107493}{12005}k^{2} \! - \! \frac{86860273454}{1452605}k \! + \!
\frac{8029312488}{7014007} \right) \\
\times& \, (3k \! + \! 4)^{2} \! \left(\frac{2}{9} \right)^{k}, \quad k \! = \! 5,6,\ldots.
\end{align*}
\end{proposition}

\begin{proof} Solving Equation~\eqref{eq:An} successively for $n \! = \! 3$, $4$, and $5$, one
derives rational solutions of these equations that vanish as $z \! \to \! 0$ faster than $z^{2}$.
This condition uniquely defines the functions $A_{3}(z)$, $A_{4}(z)$, and $A_{5}(z)$:
\begin{align*}
A_{3}(z) &= -\frac{32}{2953125}z^{7} \! - \! \frac{8752}{4134375}z^{4} \! + \!
\frac{68258}{2296875}z \! + \! \frac{116878}{459375} \, \frac{z}{(1 \! - \! 2z^{3}/9)^{2}} \\
&+ \frac{8086604}{2296875} \, \frac{z}{(1 \! - \! 2z^{3}/9)^{3}} \! - \! \frac{9771516}{765625}
\, \frac{z}{(1 \! - \! 2z^{3}/9)^{4}} \! + \! \frac{139968}{15625} \, \frac{z}{(1 \! - \! 2z^{3}/9)^{5}},
\end{align*}
\begin{align*}
A_{4}(z) &= \frac{256}{13953515625}z^{12} \! - \! \frac{78928}{45581484375}z^{9} \! - \!
\frac{17694848}{1838453203125}z^{6} \! - \! \frac{330698309}{81709031250}z^{3} \\
&+ \frac{61927956}{3242421875} \! + \! \frac{48238574611}{453939062500} \,
\frac{1}{(1 \! - \! 2z^{3}/9)} \! - \! \frac{1400615705869}{453939062500} \,
\frac{1}{(1 \! - \! 2z^{3}/9)^{2}} \\
&+ \frac{1407265401}{2316015625} \, \frac{1}{(1 \! - \! 2z^{3}/9)^{3}} \! + \!
\frac{59441369643}{1875781250} \, \frac{1}{(1 \! - \! 2z^{3}/9)^{4}} \\
&- \frac{1024460784}{19140625} \, \frac{1}{(1 \! - \! 2z^{3}/9)^{5}} \! + \!
\frac{1889568}{78125} \, \frac{1}{(1 \! - \! 2z^{3}/9)^{6}},
\end{align*}
\begin{align*}
A_{5}(z) &= \frac{131072}{197791083984375}z^{14} \! + \! \frac{90116032}{564070869140625}
z^{11} \! - \! \frac{324630499328}{23302394349609375}z^{8} \\
&+ \frac{4366976622}{9785166015625}z^{5} \! - \! \frac{385406999424}{68496162109375}
z^{2} \! + \! \frac{4531503785253}{479473134765625} \, \frac{z^{2}}{(1 \! - \! 2z^{3}/9)} \\
&+ \frac{59545228803909}{479473134765625} \, \frac{z^{2}}{(1 \! - \! 2z^{3}/9)^{2}} \! - \!
\frac{37276082380518}{13699232421875} \, \frac{z^{2}}{(1 \! - \! 2z^{3}/9)^{3}} \\
&+ \frac{14383449268992}{2837119140625} \, \frac{z^{2}}{(1 \! - \! 2z^{3}/9)^{4}} \! + \!
\frac{268395996744}{23447265625} \, \frac{z^{2}}{(1 \! - \! 2z^{3}/9)^{5}} \\
&- \frac{13329012672}{478515625} \, \frac{z^{2}}{(1 \! - \! 2z^{3}/9)^{6}} \! + \!
\frac{136048896}{9765625} \, \frac{z^{2}}{(1 \! - \! 2z^{3}/9)^{7}}.
\end{align*}
Expanding these functions at $z \! = \! 0$ into Taylor series and comparing the respective
expansions with the corresponding series in terms of the coefficients $p_{m,1}$, that is,
\begin{gather*}
A_{3}(z) \! = \! \sum_{k=1}^{\infty}p_{3k+1,1}z^{3k+1}, \qquad A_{4}(z) \! = \!
\sum_{k=2}^{\infty}p_{3k,1}z^{3k}, \qquad A_{5}(z) \! = \! \sum_{k=2}^{\infty}p_{3k+2,1}
z^{3k+2},
\end{gather*}
one arrives at the explicit expressions for the coefficients $p_{m,1}$ stated in the proposition.
\end{proof}
\begin{remark}
Although the functions $A_{n}(z)$, $n \! = \! 2,3,4,5$, derived above look rather cumbersome,
one can easily continue to calculate the functions $A_{n}(z)$ for values of $n \! > \! 5$ with the help of
{\sc Maple/Mathematica}. These functions are uniquely defined as the rational solutions of
equation~\eqref{eq:An} with $A_{n}(z) \! \to \! z^{k}$ as $z \! \to \! 0$ for $k \! > \! 2$. The explicit
value of $k$ for each $n$ is not important for the unique determination of $A_{n}(z)$$;$ rather, it is
sufficient to require that $k \! > \! 2$. The calculation, via {\sc Maple}, of the functions $A_{n}(z)$
for the values of $n$ presented above takes $3$ to $4$ seconds, so that, as a practical matter, this
procedure can be continued for much larger values of $k$.
\hfill $\blacksquare$\end{remark}
\section{The Content of the Polynomials $c_m(c_1)$} \label{sec:fence}
Propositions~\ref{prop:senior-coeffs} and \ref{prop:p-m1-coeffs}, together with the calculations of
the polynomials $c_{m}$ for the first several hundred values of $m$, imply the following conjecture.
\begin{conjecture}
The coefficients of the polynomials $c_{m}(c_{1})$ are positive rational numbers.
\end{conjecture}

One says that a rational number is divisible by an integer $p$ if the numerator in the representation
of the rational number as an irreducible fraction of two integers is divisible by $p$. A polynomial
with rational coefficients is said to be divisible by an integer $p$ if all its coefficients are divisible
by $p$.

According to these definitions, the polynomials $c_{m}(c_{1})$ for $m \! = \! 2,\ldots,9$ are divisible by $2$
(cf. equations~\eqref{eqs:c1-9}). Using the recurrence relation~\eqref{eq:recurrence}, one proves
that the $c_{m}(c_{1})$'s are divisible by $2$ for all $m\geqslant2$.  With the help of \textsc{Maple}, one can
formulate the following
\begin{conjecture}\label{con:2-adic-valuation}
For $m \! \geqslant \! 2$,
\begin{equation}
\mathrm{g.c.d.} \; c_{m}(c_{1}) \! := \! \mathrm{g.c.d.} \lbrace p_{m,0},p_{m,1},\ldots,p_{m,r_m} \rbrace
\! = \! 2^{z_{m}},
\end{equation}
where $r_{m}$ is defined in equation~\eqref{eq:cm-structure}, and $z_{m} \! \geqslant \! 1$ is an
integer sequence.
\end{conjecture}
\begin{remark}
Conjecture~\ref{con:2-adic-valuation} can be reformulated in terms of the \emph{content} of a polynomial:
the presentation of the content of the polynomial $c_m$, $m\geqslant2$, as a factorization in primes has a positive
2-adic valuation and the other valuations are non-positive.
\hfill $\blacksquare$\end{remark}
\begin{remark}\label{rem:c-m-time}
The calculation of the polynomials $c_{m}(c_{1})$ via {\sc Maple} on a notebook computer
with 16Gb RAM and processor Intel Core {\rm i7-7700HQ} takes a considerable amount of time; for example, for $m=100$,
the calculation takes less than $1$ minute, for $m=200$, it takes approximately $10$ minutes, for $m=700$, it takes about
$56$ hours, and, for $m=800$, it takes about $113$ hours.\footnote{\label{foot:c-m-time} This calculation
was done in 2020: newer laptops with 12th/13th generation processors will, certainly, expedite the calculations.}
\hfill $\blacksquare$\end{remark}

The remainder of this section is dedicated to the description of the integer sequence $z_{m}$. This
description is based on numerical calculations employing \textsc{Maple}, and, therefore, will be of a conjectural
nature.

The behaviour of the sequence $z_{m}$ depends crucially on the parity of $m$:
\begin{align}
z_{793}&=4,  &z_{795}& =6,  &z_{797}&=5,  &z_{799}&=6,\label{eqs:large-values-oddfence} \\
z_{794}&=268, &z_{796}&=266, &z_{798}&=275, &z_{800}&=268.\label{eqs:large-values-evenfence}
\end{align}
These values for the sequence $z_m$ are obtained as a result of the calculation discussed in Remark~\ref{rem:c-m-time}.
\subsection{The Odd Fence}
One begins with the description of the subsequence of $z_{m}$ for odd $m$.
Set $z_{1}=0$; since the polynomial $c_{1}(c_{1})=c_{1}$, it is not divisible by $2$. Following \cite{KitSIGMA2019},
for the description of the
subsequence $z_{2k-1}$, $k\in\mathbb{N}$, it is convenient to present it in a geometric form
by describing the plot of the function $z_{2k-1}$. If one connects the neighbouring points of the
plot, that is, the points with co-ordinates $(2k-1,z_{2k-1})$ and $(2k+1,z_{2k+1})$, $k\in\mathbb{N}$, by straight-line
segments, then one gets a graph which is called the upper edge of a fence. The fence is defined as the domain of the plane
bounded by the upper edge and the $x$-axis. The fence (its upper edge) is easy to construct with the aid of the following shape:
{\center
\begin{picture}(420,80)
\put(10,40){\circle*{2}}
\put(10,40){\line(1,1){20}}
\put(30,60){\circle*{2}}
\put(30,60){\line(2,-1){20}}
\put(50,50){\circle*{2}}
\put(50,50){\line(2,1){20}}
\put(70,60){\circle*{2}}
\put(70,60){\line(2,-1){20}}
\put(90,50){\circle*{2}}
\put(90,50){\line(1,1){20}}
\put(110,70){\circle*{2}}
\put(110,70){\line(2,-1){20}}
\put(130,60){\circle*{2}}
\put(130,60){\line(2,1){20}}
\put(150,70){\circle*{2}}
\put(170,60){Shape $\mathcal{A}$ is the $8$-tuple of points connected by}
\put(170,40){straight-line segments. The respective heights}
\put(10,20){of the points are: $z,z \! + \! 2,z \! + \! 1,z \! +\! 2,z \! + \! 1,
z \! + \! 3,z \! + \! 2,z \! + \! 3$. The $x$-coordinates of the}
\put(10,0){successive (nearest neighbour) points differ by $2$.}
\end{picture}}

With the help of shape $\mathcal{A}$, the fence can be constructed as follows. Put the left end of
$\mathcal{A}$ at the point $(1,z_{1}) \! = \! (1,0)$: the right end of $\mathcal{A}$ has co-ordinates
$(15,z_{15}) \! = \! (15,3)$. The next point on the upper edge the fence has co-ordinates
$(17,z_{15}-a_{1})=(17,1)$. Then, one takes the second
sample of $\mathcal{A}$ and places its left end at the last point $(17,1)$. The co-ordinates of the
right end are $(31,4)$. The subsequent point on the upper edge of the fence has co-ordinates
$(33,z_{31}-a_{2})=(33,1)$. Then, the procedure is continued analogously. The key point
of this procedure is the definition
\begin{equation}\label{eq:drop-down-a-n}
z_{16n+1} \! = \! z_{16n-1} \! - \! a_{n}
\end{equation}
at the $n$th iteration. Symbolically, the fence can be presented as follows:
\begin{equation}\label{eq:A-symbolic}
\mathcal{A} \, a_{1} \mathcal{A} \, a_{2} \mathcal{A} \, a_{3} \mathcal{A} \, a_{4} \ldots,
\end{equation}
where $a_n$ is a symbolic presentation of equation~\eqref{eq:drop-down-a-n}, namely, the height of the fence at the point
with the $x$-coordinate $16n+1$ is lower than its height at the previous point, $16n-1$, by the amount $a_n$.
Clearly, to completely define the fence, one requires knowledge of the numbers $a_{n}$ for all $n \! \in
\! \mathbb{N}$. The sequence $a_{n}$ can be constructed inductively in the following manner. Let $a_{1}
\! = \! 2$, and assume that $a_{1},\ldots,a_{k}$, $k \! \geqslant \! 1$, are constructed; then, the next $k$
terms are $a_{1},\ldots,a_{k-1}$,\footnote{For $k \! = \! 1$, this set is empty.} and $a_{k} \! + \! 1$. The
first few members of this sequence are
\begin{equation} \label{eq:ak-def}
a_{1} \! = \! 2, \; a_{2} \! = \! 3, \; a_{3} \! = \! 2, \; a_{4} \! = \! 4, \; a_{5} \! = \! 2, \;
a_{6} \! = \! 3, \; a_{7} \! = \! 2, \; a_{8} \! = \! 5, \; \ldots.
\end{equation}
The sequence $a_{n}$ is the sequence A085058 in OEIS \cite{OEIS}.
In \cite{OEIS}, one can find other definitions for this sequence; for example, $a_{n}$ is the
number of divisors of $2n$ of the form $2^{k}$, including the trivial divisors, that is, $1=2^{0}$ and $2n$ in case it is
a power of 2. This definition can be reformulated thus: $2^{a_{n}}$ is the highest divisor of the form $2^{k}$ of the number
$4n$, i.e., the $2$-adic valuation of $4n$.

One observes that the shape $\mathcal{A}$ can be presented in the form $\mathcal{B} 1 \mathcal{B}$,
where $\mathcal{B}$ is the following $4$-tuple shape:
{\center
\begin{picture}(420,20)
\put(10,0){The shape $\mathcal{B}$ is the first half of the shape $\mathcal{A}$,}
\put(340,0){\circle*{2}}
\put(340,0){\line(1,1){20}}
\put(360,20){\circle*{2}}
\put(360,20){\line(2,-1){20}}
\put(380,10){\circle*{2}}
\put(380,10){\line(2,1){20}}
\put(400,20){\circle*{2}}
\end{picture}}
\\
and $1$ has the same meaning as any of the numbers $a_n$
in the symbolic sequence~\eqref{eq:A-symbolic}, i.e., it denotes the point at the upper edge of the fence whose height
is smaller by $1$ compared to the height of the end-point of the first shape $\mathcal{B}$.

Via the shape $\mathcal{B}$, the fence can be presented symbolically as
\begin{equation}\label{eq:B-symbolic}
\mathcal{B} \, \tilde{a}_{1} \mathcal{B} \, \tilde{a}_{2} \mathcal{B} \, \tilde{a}_{3} \mathcal{B} \,
\tilde{a}_{4} \ldots,
\end{equation}
where the sequence $\tilde{a}_{n}$ reads:
\begin{equation} \label{eq:tilde-an-def}
1,a_{1},1,a_{2},1,a_{3},1,a_{4},\ldots = \! 1,2,1,3,1,2,1,4 \ldots.
\end{equation}
This is the sequence A001511 in \cite{OEIS}. An alternative definition for the sequence $\tilde{a}_{n}$ reads:
it is the $2$-adic valuation of $2n$ for $n\in\mathbb{N}$, that is, the largest divisor of the form $2^{k}$ of $2n$.
In \cite{OEIS}, the reader 	will find many other fascinating properties and applications of the sequences
$a_{n}$ and $\tilde{a}_{n}$; in particular, if one truncates these sequences prior to the first occurrence
of the number $k \! \geqslant \! 2$, then their finite segments are palindromes.

Either one of the symbolic sequences allows us to derive the sequence $z_n$ for odd $n$. It is slightly easier to use
the symbolic representation~\eqref{eq:B-symbolic}. Let $n$ be an odd natural number; then, define
$k:=\lfloor\frac{n}{8}\rfloor$ and $p:=n\pmod{8}$, i.e., $p=1,3,5$, or $7$. For these remainders, define the function $q(p)$:
\begin{equation}\label{eq:q(p)}
q(1)=0,\quad
q(3)=2,\quad
q(5)=1,\quad
q(7)=2;
\end{equation}
then, we can write
\begin{equation}\label{eq:odd-fence-formula}
z_n=2k-\tilde{b}_k+q(p),\qquad\mathrm{where}\qquad
\tilde{b}_k=\sum_{l=1}^k\tilde{a}_l,
\end{equation}
with $\tilde{a_l}$ the sequence~\eqref{eq:tilde-an-def}. The sequence $\tilde{b}_k$, whose first few members are
\begin{equation}\label{eq:numbers-tilde-b-k}
\tilde{b}_0,\tilde{b}_1,\tilde{b}_2,\tilde{b}_3,\tilde{b}_4,\tilde{b}_5,\tilde{b}_6,\tilde{b}_7,\tilde{b}_8,\tilde{b}_9,
\tilde{b}_{10},\ldots=0, 1, 3, 4, 7, 8, 10, 11, 15, 16, 18,\ldots,
\end{equation}
is a very popular sequence in OEIS \cite{seq:A005187}: many occurrences of this sequence in various mathematical
problems, as well as different definitions and formulae for it, are known. For our purposes, it is convenient, to use
the following representation for the numbers $\tilde{b}_k$:
\begin{equation}\label{eq:tilde-b-k-formulae}
\tilde{b}_k=\sum_{m\geqslant0}\left\lfloor\frac{k}{2^m}\right\rfloor=k+\left\lfloor\frac{k}{2}\right\rfloor
+\left\lfloor\frac{k}{2^2}\right\rfloor+\left\lfloor\frac{k}{2^3}\right\rfloor+\ldots=k+\nu_2(k!),
\end{equation}
where we used Legendre's formula for the $2$-adic valuation of $k!$, namely, $\nu_2(k!)$ (see \cite{Eremin2019} for a recent
introduction).
Substituting \eqref{eq:tilde-b-k-formulae} into \eqref{eq:odd-fence-formula}, we arrive at the formula for $z_n$ for odd $n$:
\begin{equation}\label{eq:z-n-odd}
z_n=k-\sum_{m\geq1}\left\lfloor\frac{k}{2^m}\right\rfloor +q(p)=s_2(k)+q(p),\qquad
k=\left\lfloor\frac{n}{8}\right\rfloor,\quad
p=n\!\!\!\!\pmod{8},
\end{equation}
where $q(p)$ is defined in \eqref{eq:q(p)}, and $s_2(k)$ is the sum of digits of the number $k$ in the base-$2$ numeral system
(the binary representation).
Let's check this formula for the values of $z_n$ \eqref{eqs:large-values-oddfence} calculated directly with the help of {\sc Maple}:
\begin{gather*}
n=793,\;\;k=\left\lfloor\frac{793}{8}\right\rfloor=99=1100011_2,\;\;p=1,\;\;q(1)=0,\\
z_{793}=99-\left\lfloor\frac{99}{2}\right\rfloor-\left\lfloor\frac{99}{4}\right\rfloor-\left\lfloor\frac{99}{8}\right\rfloor
-\left\lfloor\frac{99}{16}\right\rfloor-\left\lfloor\frac{99}{32}\right\rfloor-\left\lfloor\frac{99}{64}\right\rfloor\\
=99-49-24-12-6-3-1=50-30-16=4.
\end{gather*}
To verify the other values of $z_n$ in \eqref{eqs:large-values-oddfence}, it suffices to add $q(3)=2$, $q(5)=1$, and
$q(7)=2$, respectively, to $z_{793}$.

A formula similar to \eqref{eq:odd-fence-formula}, but more complicated and with twice the period, can be derived from the
symbolic sequence~\eqref{eq:A-symbolic}: in the resulting formula, we would need the sequence of partial sums
\begin{equation}\label{eq:b-k-def}
b_k=\sum_{l=1}^k a_l,\quad k=0,1,2,\ldots.
\end{equation}
The first few members of this sequence are
\begin{equation}\label{eq:b-k-values}
b_0,b_1,b_2,b_3,b_4,b_5,b_6,b_7,b_9,b_{10},\ldots=
0,2,5,7,11,13,16,18,23,25,28,\ldots.
\end{equation}
This is the sequence A004134 in OEIS \cite{seq:A004134}.\footnote{\label{foot:A004134} The fact that it represents the partial
sums of the sequence A085058 is not mentioned in \cite{seq:A004134}.}

If one ignores the initial point $z_{1}=0$, then the smallest ``sticks'' of the fence have heights $1$.
These sticks are located at the points (and only at these points) with $x$-coordinates $2^{n}+1$, $n=2,3,\ldots$,
since $z_{2^{n}+1}=1$, as follows from equation~\eqref{eq:z-n-odd}. These sticks partition the fence into apportionments.
It is easy to observe that the areas of these parts are integer numbers. Denote the area of the $n$th part of the fence
by $S_{n}$, that is, $S_{n}$ for $n\geqslant2$ is a part of the fence built on the segment $[2^{n}+1,2^{n+1}+1]$. The first
element of the sequence $S_{n}$, namely, $S_{1}=5$, is built on the segment $[1,5]$. The initial terms of $S_{n}$ are as follows:
\begin{equation*}
5,6,18,44,104,240,544,\ldots.
\end{equation*}
The definition of the fence allows one to write the equation relating the areas $S_{n}$:
\begin{equation} \label{eq:Sn-sum}
S_{n+1} \! = \! 2^{n+1} \! - \! 1 \! + \! \sum_{k=1}^{n}S_{k}, \quad n \! \geqslant \! 2.
\end{equation}
Equation~\eqref{eq:Sn-sum} is valid for $n \! \geqslant \! 2$ because the first member of the sequence
$S_{n}$ is an ``irregular'' one ($z_{1} \! = \! 0$). If one uses equation~\eqref{eq:Sn-sum} with $n \!
\geqslant \! 3$ and subtracts two successive equations, then one arrives at the recurrence relation
\begin{equation} \label{eq:Sn-rec}
S_{n+1} \! = \! 2^{n} \! + \! 2S_{n}, \quad n \! \geqslant \! 3.
\end{equation}
Equation~\eqref{eq:Sn-rec}, with the initial condition $S_{3} \! = \! 18$, has the unique solution
\begin{equation}\label{eq:Sn-explicit}
S_{n}=(2n +3)2^{n-2}, \quad
n\geqslant3.
\end{equation}
We did not find the sequence~\eqref{eq:Sn-explicit} in OEIS; however, if we remove from each apportionment the
``basis'', that is, the rectangle corresponding to the sticks of the height 1 and of length $2^{n+1}+1-(2^n+1)=2^n$,
then the remaining area of the nth apportionment, $\tilde{S}_n$, is
\begin{equation}\label{eq:tildeSn}
\tilde{S}_n=S_n-2^n=(2n-1)2^{n-2},\quad
n\geq3.
\end{equation}
Starting from $n=3$, the sequence $\tilde{S}_n$ coincides with the sequence A128135 \cite{seq:A128135}.
\subsection{The Even Fence}
The behaviour of the sequence $z_{m}$ for even $m$ differs considerably from that of odd $m$. The fence
representing this behaviour can be actualized as follows.
{\center
\begin{picture}(420,100)
\put(10,50){The initial part of the even fence}
\put(340,50){Shape $\mathcal{A}^0$}
\put(240,30){\circle*{2}}
\put(240,30){\line(1,0){20}}
\put(260,30){\circle*{2}}
\put(260,30){\line(1,3){20}}
\put(280,90){\circle*{2}}
\put(280,90){\line(1,-2){20}}
\put(300,50){\circle*{2}}
\put(220,10){\line(1,0){80}}
\put(240,10){\circle*{2}}
\put(260,10){\circle*{2}}
\put(280,10){\circle*{2}}
\put(300,10){\circle*{2}}
\put(218,0){0}
\put(238,0){2}
\put(258,0){4}
\put(278,0){6}
\put(298,0){8}
\put(300,10){\vector(1,0){10}}
\put(220,10){\line(0,1){90}}
\put(220,30){\circle*{2}}
\put(220,50){\circle*{2}}
\put(220,90){\circle*{2}}
\put(210,26){2}
\put(210,46){4}
\put(210,86){8}
\put(220,90){\vector(0,1){10}}
\end{picture}}
\\
To forge ahead with the construction, one needs to define the shapes $\mathcal{A}^1$ and $\mathcal{A}^2$.
{\center
\begin{picture}(420,170)
\put(10,10){\circle*{2}}
\put(10,10){\line(2,1){20}}
\put(30,20){\circle*{2}}
\put(30,20){\line(1,1){20}}
\put(50,40){\circle*{2}}
\put(50,40){\line(2,5){20}}
\put(70,90){\circle*{2}}
\put(70,90){\line(1,-3){20}}
\put(90,30){\circle*{2}}
\put(90,30){\line(2,3){20}}
\put(110,60){\circle*{2}}
\put(110,60){\line(2,-1){20}}
\put(130,50){\circle*{2}}
\put(130,50){\line(1,1){20}}
\put(150,70){\circle*{2}}
\put(150,70){\line(1,0){20}}
\put(170,70){\circle*{2}}
\put(170,70){\line(1,1){20}}
\put(190,90){\circle*{2}}
\put(190,90){\line(1,-1){20}}
\put(210,70){\circle*{2}}
\put(210,70){\line(1,5){18}}
\put(228,160){\circle*{2}}
\put(228,160){\line(1,-3){23}}
\put(251,90){\circle*{2}}
\put(280,90){Shape $\mathcal{A}^{1} \! := \! \mathcal{A}^{1}_{l,r}$}
\end{picture}}
\\
The shape $\mathcal{A}^{1} \! := \! \mathcal{A}^{1}_{l,r}$ is a $13$-tuple of points connected by
straight-line segments. The respective heights of the consecutive points are: $z,z \! + \! 1,z \! + \! 3,
z \! + \! 3 \! + \! l,z \! + \! 2,z \! + \! 5,z \! + \! 4,z \! + \! 6,z \! + \! 6,z \! +\! 8,z \! + \! 6,z \! + \!
6 \! + \! r,z \! + \! 8$. Here, $l$ and $r$ are positive integers denoting the heights of the left and the
right towers of the shape, respectively; their $x$-coordinates are consecutive even integers. In the
figure above, $l \! = \! 5$ and $r \! = \! 9$.
{\center
\begin{picture}(420,130)
\put(10,10){\circle*{2}}
\put(10,10){\line(2,1){20}}
\put(30,20){\circle*{2}}
\put(30,20){\line(1,1){20}}
\put(50,40){\circle*{2}}
\put(50,40){\line(2,3){20}}
\put(70,70){\circle*{2}}
\put(70,70){\line(1,-2){20}}
\put(90,30){\circle*{2}}
\put(90,30){\line(2,3){20}}
\put(110,60){\circle*{2}}
\put(110,60){\line(2,-1){20}}
\put(130,50){\circle*{2}}
\put(130,50){\line(2,3){20}}
\put(150,80){\circle*{2}}
\put(150,80){\line(2,-1){20}}
\put(170,70){\circle*{2}}
\put(170,70){\line(1,1){20}}
\put(190,90){\circle*{2}}
\put(190,90){\line(1,-1){20}}
\put(210,70){\circle*{2}}
\put(210,70){\line(2,5){20}}
\put(230,120){\circle*{2}}
\put(230,120){\line(2,-3){20}}
\put(250,90){\circle*{2}}
\put(280,90){Shape $\mathcal{A}^{2} \! := \! \mathcal{A}^{2}_{m}$}
\put(150,20){\vector(0,1){30}}
\put(150,10){Middle tower with $m \! = \! 0$}
\end{picture}}
\\
The shape $\mathcal{A}^{2} \! := \! \mathcal{A}^{2}_{m}$ is a $13$-tuple of points connected by
straight-line segments. The respective heights of the consecutive points are: $z,z \! + \! 1,z \! + \! 3,
z \! + \! 6,z \! + \! 2,z \! + \! 5,z \! + \! 4,z \! + \! 7 \! + \! m,z \! + \! 6,z \! + \! 8,z \! + \! 6,z \!
+ \! 11,z \! + \! 8$. Here, $m$ is a positive integer denoting the heights of the middle tower of the
shape. This tower corresponds to the $8$th point from the left (or the $6$th point from the right).

Using the shapes $\mathcal{A}^{k}$, $k=0,1,2$, the design of the even fence can be delineated as follows:
put the left end-point of the shape $\mathcal{A}^{0}$ at the point $(2,2)$, then glue its right
end-point located at $(8,4)$ to the left end-point of $\mathcal{A}^{1}$, then glue the right
end-point of $\mathcal{A}^{1}$ located at $(32,12)$ to the left end-point $\mathcal{A}^{2}$, then the right end-point of
$\mathcal{A}^{2}$ located at $(56,20)$ to the left end-point of $\mathcal{A}^{1}$, then the right end-point of
$\mathcal{A}^{1}$ located at $(80,28)$ to the left end-point of $\mathcal{A}^{2}$, etc., so that one has an infinite sequence
of the permuting shapes $\mathcal{A}^{1}$ and $\mathcal{A}^{2}$:
\begin{equation} \label{eq:evenFenceSymbolic}
\mathcal{A}^{0} \mathcal{A}^{1}_{l_{1},r_{1}} \mathcal{A}^{2}_{m_{1}} \mathcal{A}^{1}_{l_{2},r_{2}}
\mathcal{A}^{2}_{m_{2}} \mathcal{A}^{1}_{l_{3},r_{3}} \mathcal{A}^{2}_{m_{3}} \ldots.
\end{equation}
To complete the definition of the fence~\eqref{eq:evenFenceSymbolic}, it remains to define the sequences $l_{k}$,
$r_{k}$, and $m_{k}$, $k\in\mathbb{N}$. The first few members of these sequences are:
\begin{align*}
l_{k}:& \quad 5,9,5,7,5,13,5,7,5,9,5,7,5,11,5,7,\ldots, \\
r_{k}:& \quad 9,7,13,7,9,7,11,7,9,7,17,7,9,7,11,7,\ldots, \\
m_{k}:& \quad 0,1,0,2,0,1,0,3,0,1,0,2,0,1,0,\ldots.
\end{align*}
Begin with the definition of the sequences $l_{k}$ and $r_{k}$. It is easy to observe that $l_{k} \! = \! 5$
for $k$ odd. If one deletes these odd members of $l_{k}$ and renumerates the resulting sequence, then
one arrives at the sequence $r_{k}$. Clearly, given the sequence $r_{k}$, one can uniquely restore $l_{k}$;
therefore, the sequence $r_{k}$ has to be defined. The number $r_{k}$ defines the height of the $k$th
right tower of the shape $\mathcal{A}^{1}$. This tower has $x$-coordinate $48k-18$. The formula for $r_{k}$ reads:
\begin{equation} \label{eq:rk-def}
r_{k} \! = \! 2a_{3k-1} \! + \! 3,
\end{equation}
where the sequence $a_{k}$ is defined in the paragraph above equation~\eqref{eq:ak-def}; for example,
using equation~\eqref{eq:rk-def}, one finds that $r_{27}=15$: it is the first occurrence of the number
$15$ in the sequence $r_{k}$. To verify this result numerically, one has to calculate $48\times27-18=1278$
polynomials $c_m(c_1)$.

The sequence $m_{k}$ can be presented as follows:
\begin{equation*}
0,\tilde{a}_{1},0,\tilde{a}_{2},0,\tilde{a}_{3},0,\tilde{a}_{4},0,\ldots,
\end{equation*}
where the sequence $\tilde{a}_{k}$ is defined by equation~\eqref{eq:tilde-an-def}.

The symbolic sequence~\eqref{eq:evenFenceSymbolic} can be rewritten in terms
of explicit formulae as follows: the first three values of the sequence $z_{2n}$ are the offset numbers (cf. shape
$\mathcal{A}^0$) $z_2=z_4=2$ and $z_6=8$; the sequence $z_{2n}$, beginning from $n=4$, can be presented as the nine regular
subsequences ($k=0,1,2,3,\ldots$)
\begin{align*}
z_{8+24k} &= 4 \! + \! 8k, & z_{10+24k} &= 5 \! + \! 8k, & z_{12+24k} &= 7 \! + \! 8k, & \\
z_{16+24k} &= 6 \! + \! 8k, & z_{18+24k} &= 9 \! + \! 8k, & z_{20+24k} &= 8 \! + \! 8k, & \\
z_{24+24k} &= 10 \! + \! 8k, & z_{26+24k} &= 12 \! + \! 8k, & z_{28+24k} &= 10 \! + \! 8k,&
\end{align*}
and the three singular subsequences
\begin{align*}
z_{14+48k}&=7+l_{k+1}+16k,&&z_{14+24(2k+1)}=18+16k,\\
z_{30+48k}&=10+r_{k+1}+16k,&&z_{30+24(2k+1)}=23+16k,\\
z_{22+48k}&=10+16k,&&z_{22+24(2k+1)}=19+m_{k+1}+16k,
\end{align*}
where the sequences $l_k$, $m_k$, and $r_k$ are defined above.
\section{The Monodromy Data} \label{sec:monodromy}
The goal of this section is to calculate the monodromy data for the vanishing meromorphic solutions corresponding to $a=\pm\mi/2$.
These data are used to find asymptotics of the solutions as $\tau\to+\infty$; some examples of such asymptotics are presented
in Appendix~\ref{app:pictures}.

In \cite{KitVar2004} we considered an isomonodromy deformation system whose solution is given in terms of the
pair of functions $(u(\tau),\me^{\mi\varphi(\tau)})$, where $u(\tau)$ is the degenerate third Painlev\'e function solving
equation~\eqref{eq:dp3}, and $\varphi(\tau)$ is the indefinite integral of $1/u(\tau)$. In \cite{KitVar2004} we defined the
monodromy manifold that can be presented in terms of the associated monodromy data. Consider $\mathbb{C}^{8}$ with
co-ordinates $(a,s_{0}^{0},s_{0}^{\infty},s_{1}^{\infty},g_{11},g_{12},g_{21},g_{22})$, where the parameter of
formal monodromy, $a$, the Stokes multipliers, $s_{0}^{0}$, $s_{0}^{\infty}$, and $s_{1}^{\infty}$, and the elements
of the connection matrix, $(G)_{ij} \! =: \! g_{ij}$, $i,j \! \in \! \lbrace 1,2 \rbrace$, are
called the \emph{monodromy data}. These monodromy data are related by the set of algebraic equations
\begin{gather}
s_{0}^{\infty}s_{1}^{\infty} \! = \! -1 \! - \! \me^{-2 \pi a} \! - \! \mi s_{0}^{0}
\me^{-\pi a}, \label{eq:monodromy:s} \\
g_{21}g_{22} \! - \! g_{11}g_{12} \! + \! s_{0}^{0}g_{11}g_{22} \! = \! \mi \me^{-\pi a},
\label{eq:monodromy:main} \\
g_{11}^{2} \! - \! g_{21}^{2} \! - \! s_{0}^{0} g_{11} g_{21} \! = \! \mi s_{0}^{\infty}
\me^{-\pi a}, \label{eq:monodromy:s0} \\
g_{22}^{2} \! - \! g_{12}^{2} \! + \! s_{0}^{0} g_{12} g_{22} \! = \! \mi s_{1}^{\infty}
\me^{\pi a}, \label{eq:monodromy:s1} \\
g_{11}g_{22} \! - \! g_{12} g_{21} \! = \! 1. \label{eq:monodromy:detG}
\end{gather}
The system~\eqref{eq:monodromy:s}--\eqref{eq:monodromy:detG} defines an algebraic variety which we call
the \emph{manifold of the monodromy data}, $\mathscr{M}$.
The formal monodromy $a$ enters into the system~\eqref{eq:monodromy:s}--\eqref{eq:monodromy:detG} in the form $\me^{\pi a}$.
It is, of course, possible to denote this exponent by some new variable and to arrive at a purely algebraic system; however,
it is more convenient to treat $\me^{\pi a}$ as a parameter in the system of algebraic equations because the particular
choice of this parameter means that we have chosen a particular equation~\eqref{eq:dp3}. Only four of the five equations
defining $\mathscr{M}:=\mathscr{M}_a$ are independent, so that $\dim_{\mathbb{C}}\mathscr{M}_a=3$. What is important to know is
that, to each pair of functions $(u(\tau),\me^{\mi\varphi(\tau)})$, there correspond two different points on $\mathscr{M}_a$,
that is, $(a,s_{0}^{0},s_{0}^{\infty},s_{1}^{\infty}, g_{11},g_{12},g_{21},g_{22})$ and
$(a,s_{0}^{0},s_{0}^{\infty},s_{1}^{\infty}, -g_{11},-g_{12},-g_{21},-g_{22})$. This occurs because the definition of
the canonical solution $X_k^0(\mu)$ at $\mu=0$ (see \cite{KitVar2004}, p.1170, equations (16) and (19)) contains the term
$\sqrtsign{B(\tau)}$, where the function $B(\tau)$ is one of the coefficient functions defining the isomonodromy deformations
of the auxiliary linear matrix ODE. The branch of this square root can be chosen arbitrarily, while
the isomonodromy deformation system remains unchanged; however, the connection matrix, $G$, changes sign, $G\to-G$.
The manifold $\mathscr{M}_a$ is convenient to use when working with the system of isomonodromy deformations, since all
its co-ordinates have a clearly recognizable sense in terms of the monodromy data; however, when there is no need to address
these data, it may be more suitable to use a manifold that uniquely parametrizes the corresponding pair of functions
$(u(\tau),\me^{\mi\varphi(\tau)})$: we call such a manifold the \emph{contracted monodromy manifold},
$\widetilde{\mathscr{M}}_a$. It can be defined as follows:
\begin{gather*}
\tilde{g}_1:=\mi g_{12}g_{11},\quad
\tilde{g}_2:=\mi g_{21}g_{22},\quad
\tilde{g}_3:=g_{11}g_{22},\quad
\tilde{g}_4:=g_{12}g_{21},\\
\tilde{s}:=1+\mi s_0^0,\quad
\tilde{f}_1:=g_{12}^2,\quad
\tilde{f}_2:=g_{21}^2.
\end{gather*}
The tilde-variables satisfy the following system of algebraic equations:
\begin{gather*}
\tilde{g}_1-\tilde{g}_2+(1-\tilde{s})\tilde{g}_3=\me^{-\pi a},\quad
\tilde{g}_3\tilde{g}_4=-\tilde{g}_1\tilde{g}_2,\quad
\tilde{g}_3-\tilde{g}_4=1,\quad
\tilde{f}_1\tilde{f}_2=\tilde{g}_4^2.
\end{gather*}
One can further exclude the variables $\tilde{g}_2$ and $\tilde{g}_4$ in order to consider $\widetilde{\mathscr{M}}_a$
as embedded in $\mathbb{C}^5$:
\begin{equation}\label{eqs:contractedManifold}
\tilde{g}_3^2+\tilde{g}_1^2+(1-\tilde{s})\tilde{g}_1\tilde{g}_3=\tilde{g}_3+\tilde{g}_1\me^{-\pi a},\quad
\tilde{f}_1\tilde{f}_2=(\tilde{g}_3-1)^2.
\end{equation}
This represents a separation of the monodromy variables; the first equation in \eqref{eqs:contractedManifold} defines the
function $u(\tau)$, and, subsequently, the second equation defines the constant of integration of the associated indefinite
integral in terms of the monodromy variables. The last fact is helpful for finding various definite integrals of $1/u(\tau)$
\cite{KitVar2019,KitVar2023}. In \cite{KitVar2023} we showed that the projectivization of the first
equation in \eqref{eqs:contractedManifold} is a singular cubic surface with a singularity of type $A_3$.

The aforementioned comments regarding the monodromy manifold concern the generic setting for the parameters,
$a\in\mathbb{C}$, $\varepsilon=\pm1$, and $\varepsilon b\in\mathbb{R}_+$. For the remainder of this section, we assume that
$a=\pm\mi/2$, and the normalization $a=b$ is no longer assumed.

By using the change of variables~\eqref{eq:changevar}, we rewrite the expansion~\eqref{eq:taylor-expansion} in terms of the
original variables in \eqref{eq:dp3}:
\begin{equation} \label{eq:uSuleimanov:general}
u(\tau)\underset{\tau\to0}{=} -\frac{b \tau}{2a}\!\left(\tilde{c}_{0}+\tilde{c}_{1}\tau+
\sum_{m=2}^{\infty}\tilde{c}_{m}\tau^{m}\right),
\end{equation}
where $\tilde{c}_{0}=1$, and $\tilde{c}_{1}$ is a complex parameter. For  $a=\kappa\mi/2$, $\kappa=\pm1$, the coefficients
$\tilde{c}_m$, $m\geqslant2$, can be determined uniquely via the recurrence relation
\begin{equation} \label{eq:tilde-c-m-recurrence}
(m^2-1)\tilde{c}_{m}=\sum_{k=0}^{m-2}(k+2)(m-2(k+1))\tilde{c}_{k+1}\tilde{c}_{m-k-1}
-8\kappa\,\varepsilon b\,\mi\sum_{k=0}^{m-2}\sum_{l=0}^{k}\tilde{c}_{m-k-2}\tilde{c}_{l}\tilde{c}_{k-l}.
\end{equation}
In terms of the coefficients $c_m$ studied in Sections~\ref{sec:Taylor}--\ref{sec:fence}, the coefficients $\tilde{c}_m$ are
given by the following formula:
\begin{equation}\label{eq:c-m-tilde}
\tilde{c}_m=\alpha^m c_m.
\end{equation}
The first few members of this sequence can be written as follows:
\begin{equation}\label{eq:c2-c4tilde}
C:=-8\kappa\,\varepsilon b\,\mi,\quad
\tilde{c}_2=\frac{C}{3},\quad
\tilde{c}_3=\frac{C}{3}\tilde{c}_1,\quad
\tilde{c}_4=\frac{C}{3}\left(\frac{\tilde{c}_1^2}{3}+\frac{C}{5}\right).
\end{equation}
\begin{lemma} \label{lem:mondata-u-pm-id2}
Let $u(\tau)$ denote solutions of equation~\eqref{eq:dp3} for $\varepsilon b>0$, $a=\kappa\mi/2$, $\kappa=\pm1$,
defined by the asymptotic expansion~\eqref{eq:uSuleimanov:general}. Then, the monodromy data of these solutions reads:
\begin{equation} \label{eqs:i/2mondata}
\begin{gathered}
\pmb{(\mi)}\quad
a=\frac{\mi}{2},\quad
s_{0}^{0}=s_{1}^{\infty}=0,\quad
s_{0}^{\infty}g_{12}^2=\frac{\sqrt{\pi}\me^{\mi\pi/4}\tilde{c}_1}{2^{3/2}(\varepsilon b)^{1/2}}\equiv X\me^{\mi\pi/4},\quad
g_{22}=-g_{12}\neq0,\\
\tilde{g}_1=\mi g_{12}g_{11}=-\frac{\mi}{2}(1+X\me^{\mi\pi/4}),\quad
\tilde{g}_2=\mi g_{21}g_{22}=\frac{\mi}{2}(1-X\me^{\mi\pi/4}),\quad\\
\tilde{g}_3=g_{11}g_{22}=\frac12(1+X\me^{\mi\pi/4}),\quad
\tilde{g}_4=g_{12}g_{21}=-\frac12(1-X\me^{\mi\pi/4});
\end{gathered}
\end{equation}
and
\begin{equation}\label{eqs:-i/2mondata}
\begin{gathered}
\pmb{(\mi\mi)}\quad
a=-\frac{\mi}{2},\quad\!
s_{0}^{0}=s_{0}^{\infty}\!=0,\quad\!
s_{1}^{\infty}g_{21}^{2}=\frac{\sqrt{\pi}\me^{-\mi\pi/4}\tilde{c}_1}{2^{3/2}(\varepsilon b)^{1/2}}\equiv X\me^{-\mi\pi/4},\quad\!
g_{11}=-g_{21}\neq0, \\
\tilde{g}_1=\mi g_{12}g_{11}=\frac{\mi}{2}(1-X\me^{-\mi\pi/4}),\quad
\tilde{g}_2=\mi g_{21}g_{22}=-\frac{\mi}{2}(1+X\me^{-\mi\pi/4}),\\
\tilde{g}_3=g_{11}g_{22}=\frac12(1+X\me^{-\mi\pi/4}),\quad
\tilde{g}_4=g_{12}g_{21}=-\frac12(1-X\me^{-\mi\pi/4}).\\
\end{gathered}
\end{equation}
\end{lemma}
\begin{proof}
The proof of this lemma relies substantially on Theorem~3.4 of \cite{KitVar2004}; however, in this proof, we refer to
Theorem~B.1 of our recent work \cite{KitVar2023}, where we simplified the notation and some
formulae.\footnote{\label{foot:theorem3.4} The reader can use the results stated in Theorem 3.4 of \cite{KitVar2004}
to arrive at the same formulae.}
Consider, say, the proof for the case $a=+\mi/2$, since the proof for
$a=-\mi/2$ is, \emph{mutatis mutandis}, similar. According to Theorem~B.1 of \cite{KitVar2023}, a solution $u(\tau)$
corresponding to a generic set of the monodromy data has the following small-$\tau$ asymptotic behaviour:
\begin{align} \label{eq:u-asympt0-generic}
u(\tau)\underset{\tau\to0}{=}&\,\frac{\tau b\me^{\pi a/2}}{16\pi}\!\left(
\varpi_{1}(-\varrho)\varpi_{2}(-\varrho)\tau^{-4\varrho}+\varpi_{1}(-\varrho)\varpi_{2}(\varrho)\right.
\nonumber \\
+&\left. \, \varpi_{1}(\varrho)\varpi_{2}(-\varrho)+\varpi_{1}(\varrho)\varpi_{2}(\varrho)\tau^{4\varrho}\right)\!
\left(1+o(\tau^{\delta})\right),
\end{align}
where $\delta>0$, $a=\mi/2$, the parameter $\varrho\neq0$, with $|\mathrm{Re}(\varrho)|<1/2$,
and the $\tau$-independent functions $\varpi_{k}(\ast)$, $k=1,2$, are expressed in terms
of the monodromy data by equations (B.3) and (B.6)--(B.8), respectively, of \cite{KitVar2023}. The other possible asymptotic
behaviours as $\tau\to0$ stated in the original paper \cite{KitVar2004} are, obviously, incompatible with the
expansion~\eqref{eq:uSuleimanov:general}.\footnote{\label{foot:theorem3.5ofcite2004} For $\varrho=0$,
we refer to Theorem~3.5 of \cite{KitVar2004}, which states that there exists a one-parameter family of
solutions with logarithmic behaviour as $\tau\to0$: the requirement for the absence of logarithmic terms contradicts
equation~\eqref{eq:monodromy:detG}. For $\varrho=1/2$, $1/u(\tau)\sim-4\varepsilon\tau\ln^2\tau$ as $\tau\to0$.}
The asymptotic formula~\eqref{eq:u-asympt0-generic} is invariant under the reflection
symmetry $\varrho\to-\varrho$; therefore, without loss of generality, we assume that $0<\mathrm{Re}(\varrho)<1/2$.
\begin{remark}\label{rem:lemma-proof}
In this remark, we digress from the proof in order to explain the problem associated with the application of the
asymptotics~\eqref{eq:u-asympt0-generic} for finding a relation between the parameter $\tilde{c}_1$ and the monodromy data.
The asymptotics~\eqref{eq:u-asympt0-generic} was obtained in \cite{KitVar2004} for finding a relation between the coefficients
of the local expansion of the general solution of equation~\eqref{eq:dp3} and the monodromy data. For that purpose, the value
of the parameter $\delta>0$ in the error estimate of \eqref{eq:u-asympt0-generic} is not important, because all the necessary
parameters are contained in the leading term of the corresponding local expansion. An example of the application of
\eqref{eq:u-asympt0-generic} in such a ``standard situation'' is discussed in detail in \cite{KitVar2023}.
In the context of the present study, though, the situation is more complicated. In order for the asymptotic expansions
\eqref{eq:uSuleimanov:general} and \eqref{eq:u-asympt0-generic} to match, the following system must be valid:
\begin{equation}\label{sys: asympt0-degeneration-general}
\varpi_{1}(-\varrho)\varpi_{2}(-\varrho)=0,\quad
\varpi_{1}(-\varrho)\varpi_{2}(\varrho)+\varpi_{1}(\varrho)\varpi_{2}(-\varrho)=16\pi\me^{\pi\mi/4}.
\end{equation}
The system~\eqref{sys: asympt0-degeneration-general} is equivalent to the fact that at least one of the following two systems
is satisfied:
\begin{gather}
\varpi_{1}(-\varrho)=0,\quad
\varpi_{1}(\varrho)\varpi_{2}(-\varrho)=16\pi\me^{\pi\mi/4},\label{sys:degenerate1}\\
\varpi_{2}(-\varrho)=0,\quad
\varpi_{1}(-\varrho)\varpi_{2}(\varrho)=16\pi\me^{\pi\mi/4}\label{sys:degenerate2}.
\end{gather}
Consider the system~\eqref{sys:degenerate1}: the first equation (cf. equations (B.6)--(B.8) of \cite{KitVar2023})
implies that
\begin{equation}\label{eq:g11g21}
g_{11}=-g_{21}\me^{2\pi\mi(\varrho-1/4)}.
\end{equation}
Substituting, now, \eqref{eq:g11g21} into the second equation of the system~\eqref{sys:degenerate1}, where
$\varpi_{2}(-\varrho)$ is given by equations (B.6)--(B.8) of \cite{KitVar2023},
and using \eqref{eq:monodromy:detG}, we remove all the $g_{ik}$ variables from this equation and arrive at an equation
that contains only the parameter $\varrho$:
\begin{equation}\label{eq:varrho1}
\frac{\sin(2\pi\varrho)}{2\varrho}\frac{(\varrho-1/4)}{\sin(\pi(\varrho-1/4))}=-4\varrho\me^{-\pi\mi(\varrho+1/4)}.
\end{equation}
Recall that equation~\eqref{eq:varrho1} is derived under the assumption $a=+\mi/2$ and we are interested in solutions
for which $0<\mathrm{Re}\,\varrho<1/2$: the obvious solution $\varrho=-1/4$ does not, therefore, fit into our construction.
The second ratio on the left-hand side of equation~\eqref{eq:varrho1} originates from the corresponding identity for
the $\Gamma$-function, and, therefore,  for $\varrho=1/4$, must be understood in the limiting sense, i.e., it equals
$1/\pi$; thus, $\varrho=1/4$ is also not a root of this equation. Our numerical studies using {\sc Maple} allow us to find
only one---mysterious!---root of equation~\eqref{eq:varrho1} in the interval $0<\mathrm{Re}\,\varrho<1/2$:
\begin{equation}\label{eq:root-varrho1}
\varrho_1=0.30116884436547816\ldots -\mi\,0.1989138937847074\ldots.
\end{equation}
Now, consider the system~\eqref{sys:degenerate2}. The solution is quite analogous; in this case, instead of
equations \eqref{eq:g11g21} and \eqref{eq:varrho1}, we get
\begin{gather}
g_{22}=-g_{12}\me^{-2\pi\mi(\varrho-1/4)},\label{eq:g22g12}\\
\frac{\sin(2\pi\varrho)}{2\varrho}\frac{(\varrho+1/4)}{\sin(\pi(\varrho+1/4))}=4\varrho\me^{\pi\mi(\varrho-1/4)}.
\label{eq:varrho2}
\end{gather}
Equation~\eqref{eq:varrho2} differs from \eqref{eq:varrho1} by the reflection $\varrho\to-\varrho$; however, the condition
$0<\mathrm{Re}\,\varrho<1/2$ holds for the last equation, too. Thus, $\varrho=1/4$ is now the ``legitimate'' root:
had we known \emph{a priori} that $\varrho$ is real, then we would conclude that it is the only root. Our numerical studies
do not indicate that any other roots of equation~\eqref{eq:varrho2} for $\mathrm{Re}\,\varrho\in(0,1/2)$ exist.
We found two roots,
\begin{equation}\label{eq:extra-roots}
\varrho_2=0.75580947\ldots-\mi0.06115553\ldots\quad\text{and}\quad
\varrho_3=1.20069834\ldots+\mi\,0.35281941\ldots,
\end{equation}
which do not belong to the set of admissible values for $\varrho$.
After this step, we can not proceed with the analysis of equation~\eqref{eq:u-asympt0-generic} because of the
indeterminacy of $\delta>0$ in the error estimate; we can not even conclude that $\varrho$ is real!

Our experience, however, allows us to formulate a conjecture that considerably clarifies the application of
the asymptotics~\eqref{eq:u-asympt0-generic} for determining the monodromy data: \emph{two leading terms of the three
explicitly presented in \eqref{eq:u-asympt0-generic} give rise to the first two leading-order terms of the
asymptotics for $u(\tau)$, and the error estimate, $o(\tau^{\delta})$, can be improved to a sharper estimate, like, say,
$\mathcal{O}\big(\tau^{8\varrho}\big)$}.\footnote{\label{foot:conjecture-asympt0} It seems that, in order to verify this
conjecture, one has to follow the corresponding proof presented in \cite{KitVar2004}, paying closer attention to the
error estimates. We hope to check this conjecture soon.}

Assume that the conjecture formulated above is true. Then, $\varrho=1/4$ (note that, in this case, $\delta=8\times1/4=2$),
and the system~\eqref{sys:degenerate2} is satisfied, so that, according to equation~\eqref{eq:g22g12}, $g_{22}=-g_{12}$,
after which, according to equation~\eqref{eq:monodromy:detG}, $g_{12}\neq0$. Now, equation (B.3) of \cite{KitVar2023} implies
that $s_0^0=-s_0^{\infty}s_1^{\infty}=0$. Substituting $s_0^0=0$ into equations~\eqref{eq:monodromy:s0} and
\eqref{eq:monodromy:s1}, we find that $s_0^{\infty}=g_{11}^2-g_{21}^2$ and $s_1^{\infty}=0$. A staightforward
calculation using equation~\eqref{eq:monodromy:detG} shows that $s_0^{\infty}g_{12}^2=g_{11}g_{22}+g_{12}g_{21}$.
According to the conjecture, we must put
\begin{equation}\label{eq:monodromy+i/2calc}
\varpi_1(1/4)\varpi_2(1/4)=16\pi\me^{\pi\mi/4}\quad
\Rightarrow\quad
32(2\pi\varepsilon b)^{1/2}(g_{11}g_{22}+g_{12}g_{21})=16\pi\me^{\pi\mi/4}.
\end{equation}
The last equation in \eqref{eq:monodromy+i/2calc} is obtained with the help of equations (B.6)--(B.8) of \cite{KitVar2023}.
It implies the formula for $s_0^{\infty}g_{12}^2$ stated in \eqref{eqs:i/2mondata}. To derive the remaining formulae for
the monodromy data in the list \eqref{eqs:i/2mondata}, one has to supplement the last equation
in \eqref{eq:monodromy+i/2calc} with \eqref{eq:monodromy:detG}, and solve the resulting linear system with respect
to $g_{11}g_{22}$ and $g_{12}g_{21}$.
\hfill $\blacksquare$\end{remark}
We now continue with the proof. In view of Remark~\ref{rem:lemma-proof}, until such time as a better estimate for $\delta$
in the asymptotics~\eqref{eq:u-asympt0-generic} is made available, we can only use it in the case where
$\varpi_{1}(-\varrho)\varpi_{2}(-\varrho)\neq0$. Towards this end, we consider a B\"acklund transformation
relating solutions of equation~\eqref{eq:dp3} for $a=\mi/2$ and $a=-\mi/2$. The B\"acklund transformation for
equation~\eqref{eq:dp3} was obtained by Gromak \cite{Gromak1973}; we, however, need a realization of this transformation as
the Schlesinger transformation of the corresponding Fuchs-Garnier pair in order to make the connection with the monodromy data
of the solutions. This realization is done in Subsection 6.1 of \cite{KitVar2004}.

Denote by $\hat{u}(\tau)$ a solution of equation~\eqref{eq:dp3} for $a=-\mi/2$: its monodromy data is also labeled `hats'.
For $a=\mi/2$, the B\"acklund transformation reads
\begin{equation}\label{eq:Baecklund-solutions}
\hat{u}(\tau)=\frac{\varepsilon b\tau}{8u^2(\tau)}\left(\mi u'(\tau)+b\right).
\end{equation}
This transformation acts on the monodromy manifold as follows:
\begin{equation}\label{eq:Baecklund-monodromy}
(\mi/2, s_0^0, s_0^{\infty} , s_1^{\infty} , g_{11}, g_{12}, g_{21}, g_{22})\rightarrow
(-\mi/2,-s_0^0, s_0^{\infty} , s_1^{\infty} , \mi g_{11}, \mi g_{12},-\mi g_{21},-\mi g_{22}).
\end{equation}
This action can also be considered as the mapping $\mathscr{M}_{\mi/2}\to\mathscr{M}_{-\mi/2}$.
Our notational agreement means that $\hat{s}_0^0=-s_0^0$, $\hat{g}_{11}=\mi g_{11}$, etc.

Substituting the expansion~\eqref{eq:uSuleimanov:general} into \eqref{eq:Baecklund-solutions}, one finds
\begin{equation}\label{eq:hat-u-neq0-at0}
\hat{u}(\tau)\underset{\tau\to0}{=}\varepsilon\left(\frac14\tilde{c}_1+
\left(\frac{3}{8}\tilde{c}_2-\frac12\tilde{c}_1^2\right)\tau+
\left(\frac12\tilde{c}_3+\frac34\tilde{c}_1^3-\frac54\tilde{c}_1\tilde{c}_2\right)\tau^2+\ldots\right),
\end{equation}
where the coefficients $\tilde{c}_k$ are given by equations~\eqref{eq:c-m-tilde} and \eqref{eq:c2-c4tilde} with $\kappa=+1$!
Note that the monodromy data for the case $\tilde{c}_1=0$, the Suleimanov solution, was calculated in \cite{KitSIGMA2019};
in this case, equations~\eqref{eqs:i/2mondata} are confirmed. Until the end of the proof, we assume that $\tilde{c}_1\neq0$.
Under this assumption, the function $\hat{u}(\tau)$ defined by the Taylor-series expansion~\eqref{eq:hat-u-neq0-at0} is the
meromorphic solution of equation~\eqref{eq:dp3} that is non-vanishing at the origin. Now, we can equate the leading terms
as $\tau\to0$ of the asymptotics~\eqref{eq:u-asympt0-generic} and the expansion~\eqref{eq:hat-u-neq0-at0} for $a=-\mi/2$:
\begin{equation}\label{eq:tilde-c1-varpi}
\varrho=1/4,\qquad
\frac{\varepsilon\tilde{c}_1}{4}=\frac{b\me^{-\pi\mi/4}}{16\pi}\varpi_1(-1/4)\varpi_2(-1/4).
\end{equation}
Substituting for $\varpi_1(-1/4)$ and $\varpi_2(-1/4)$ in the second equation of \eqref{eq:tilde-c1-varpi} the corresponding
expressions (B.6)--(B.8) of \cite{KitVar2023}, we, after straightforward simplifications, find that
\begin{equation}\label{eq:tilde-c1-mondata-hat}
\tilde{c}_1=2^{1/2}\pi^{-1/2}(\varepsilon b)^{1/2}\me^{-\pi\mi/4}(\hat{g}_{11}+\hat{g}_{21})(\hat{g}_{12}+\hat{g}_{22}).
\end{equation}
Equation (B.3) of \cite{KitVar2023}, together with the condition $\varrho=1/4$, implies that
$\hat{s}_0^0=\hat{s}_0^{\infty}\hat{s}_1^{\infty}=0$. Substituting $\hat{s}_0^0=0$ and $a=-\mi/2$ into
equation~\eqref{eq:monodromy:main} and summing it up with equation~\eqref{eq:monodromy:detG}, we can factor the resulting
equation:
\begin{equation}\label{eq:hat-g-factor}
(\hat{g}_{22}-\hat{g}_{12})(\hat{g}_{11}+\hat{g}_{21})=0.
\end{equation}
Since, according to \eqref{eq:tilde-c1-mondata-hat}, $\hat{g}_{11}+\hat{g}_{21}\neq0$, we obtain
\begin{equation}\label{eq:hat-g22g12}
\hat{g}_{22}=\hat{g}_{12};
\end{equation}
thus, $\hat{s}_1^{\infty}=0$ and $\hat{s}_0^{\infty}=\hat{g}_{21}^2-\hat{g}_{11}^2$ (see equations~\eqref{eq:monodromy:s1}
and \eqref{eq:monodromy:s0}, respectively). To complete the proof of equations~\eqref{eqs:i/2mondata}, it is enough to apply
the following identities
\begin{gather*}
\hat{s}_0^{\infty}\hat{g}_{12}^2=\hat{g}_{12}^2(\hat{g}_{21}^2-\hat{g}_{11}^2)=
-(\hat{g}_{11}^2\hat{g}_{22}^2-\hat{g}_{12}^2\hat{g}_{21}^2)=-(\hat{g}_{11}\hat{g}_{22}+\hat{g}_{12}\hat{g}_{21}),\\
(\hat{g}_{11}+\hat{g}_{21})(\hat{g}_{12}+\hat{g}_{22})=2\hat{g}_{22}(\hat{g}_{11}+\hat{g}_{21})=
2(\hat{g}_{11}\hat{g}_{22}+\hat{g}_{12}\hat{g}_{21}),\\
\hat{s}_0^{\infty}\hat{g}_{12}^2=-s_0^{\infty}g_{12}^2,\qquad
\hat{g}_{1k}\hat{g}_{2l}=g_{1k}g_{2l},\quad
k,l=1,2,
\end{gather*}
to equations~\eqref{eq:tilde-c1-mondata-hat} and \eqref{eq:monodromy:detG}.
\end{proof}
\begin{remark}\label{rem:g11g22=0}
Strictly speaking, in Theorem B.1 of \cite{KitVar2023}, as well as in Theorem 3.4 of \cite{KitVar2004}, monodromy data
for which $g_{11}g_{22}=0$ are excluded; however, the formulae \eqref{eqs:i/2mondata} and \eqref{eqs:-i/2mondata} remain valid
even for those values of $\tilde{c}_1$ such that $g_{11}g_{22}=0$. This follows from the fact that the correspondence between
$\tilde{c}_1$ and the monodromy data is given by a one-to-one analytic map.
\hfill $\blacksquare$\end{remark}
\appendix
\section{Appendix: Examples of Generic Asymptotic Behaviour as $\tau\to+\infty$ of Meromorphic Solutions for $a=\pm\mi/2$}
\label{app:pictures}
Here, we present some examples of the application of Lemma~\ref{lem:mondata-u-pm-id2} for describing the large-$\tau$
asymptotics of solutions of equation~\eqref{eq:dp3} for $a=\pm\mi/2$. Generic asymptotic behaviour means that the monodromy
data satisfy the condition $g_{11}g_{12}g_{21}g_{22}\neq0$ and $|g_{11}g_{22}|\neq-g_{11}g_{22}$. In this case, large-$\tau$
asymptotics for the solution $u(\tau)$ can be found with the help of Theorems~C.1 and C.2 of \cite{KitVar2023}. Note that,
according to Lemma~\ref{lem:mondata-u-pm-id2}, for $a=\mi/2$, there are two values of the parameter $\tilde{c}_1$
when either $g_{11}=0$ or $g_{21}=0$; and, for $a=-\mi/2$, there are also two values of $\tilde{c}_1$ when
either $g_{22}=0$ or $g_{12}=0$. For such values of $\tilde{c}_1$ large-$\tau$ asymptotics are also known; they correspond
to the truncated behaviour of the solutions, that is, solutions that admit pole-free analytic continuation into certain
sectors of the neighbourhood of the point at infinity: we do not consider such solutions here.
Thus, holomorphicity of solutions at $\tau=0$ imposes rather weak restrictions
on their behaviour at the point at infinity: if $a=\mi/2$, then $g_{22}=-g_{12}\neq0$, and, if $a=-\mi/2$, then
$g_{11}=-g_{21}\neq0$.

The calculations presented below are performed with {\sc Maple} 2022. Our initial data are specified at $\tau=0$,
where equation~\eqref{eq:dp3} has a singularity. In this version of {\sc Maple}, a standard command for solving ODEs
correctly determines that our initial conditions satisfy a corresponding algebraic constraint, and thus automatically
solves equation~\eqref{eq:dp3} without necessitating any additional tricks on the part of the end user.
At the same time, the same program that worked with {\sc Maple} 2022 did not work with {\sc Maple} 2017, because
{\sc Maple} 2017 was not able to correctly determine the special constraint mentioned above; stranger still is the fact that
the program works flawlessly with {\sc Maple} 16! In the case of problems like we had with {\sc Maple} 2017, the reader is
encouraged to see how we dealt with them in \cite{KitVar2023}.

Hereafter, we consider three examples of generic asymptotics: the first two examples correspond to $a=\mi/2$, while the last
example corresponds to $a=-\mi/2$. In all three examples, we set $\varepsilon=+1$ and $b=1/2$. In each of the figures below,
we present either the real part of the solution and its corresponding asymptotics, or the imaginary part of the solution and
its corresponding asymptotics. The real and imaginary parts of the solutions are obtained via standard {\sc Maple} commands
for the graphical presentation of the solutions of ODEs, whilst the corresponding asymptotics are obtained by substituting the
monodromy data from Lemma~\ref{lem:mondata-u-pm-id2} into the asymptotics given in Theorems C.1 and C.2 of \cite{KitVar2023}.
Since we consider large-$\tau$ asymptotics, we're supposed to compare solutions and asymptotics over large segments of $\tau$.
On large-$\tau$ segments, however, the plots are automatically scaled in order to fit a standard page, and, as a result of
this scaling, two plots merge into one curve, so that,
sometimes, even coloured plots are hard to distinguish; therefore, we choose the range of $\tau$ over which we compare our
solutions with the asymptotics such that, on the one hand, it is clear that both plots are close enough, and, on the other hand,
one can see a clear distinction between the plots. In general, the correspondence between the solution and its asymptotics is
worse over the same distances for large values of $\tilde{c}_1$ than for smaller values of $\tilde{c}_1$; therefore, when
choosing larger values of $\tilde{c}_1$, one has to consider larger segments of $\tau$
for the comparison of the plots in order to observe that they actually approach one another. The parameter $b$ in
equation~\eqref{eq:dp3} is responsible for the scaling of $\tau$, so that, instead of considering large distances, one can
simply increase the value of $b$.

The domains of validity of the asymptotics given in Theorems C.1 and C.2 of \cite{KitVar2023} overlap. Which of these
asymptotics better approximates a particular solution is regulated by the value of the parameter
$\nu+1=\frac{\mi}{2\pi}\ln(g_{11}g_{22})$, more precisely, by the value of $\mathrm{Re}\,\nu+1$. In case $\nu+1$ is purely
imaginary, the asymptotics given in Theorem C.2 is not applicable; so, in the first example (see
Figs.~\ref{fig:=dp3+id2u0+1v0-1l50Reu} and \ref{fig:dp3+id2u0+1v0-1l100Imu}), we have chosen $\tilde{c}_1=1-\mi$ such that
$\nu+1=\mi0.0189800\ldots$.
\begin{figure}[htpb]
\begin{center}
\includegraphics[height=50mm,width=100mm]{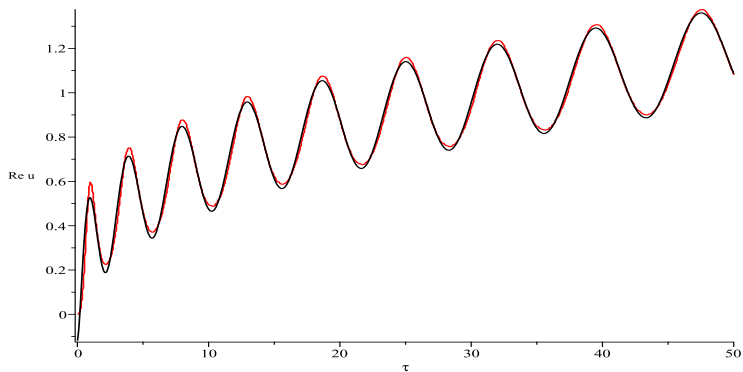}
\caption{The red and black plots are, respectively, the real parts of the numerical and large-$\tau$ asymptotic values
of the function $u(\tau)$ for $\tau\in(0,50)$ corresponding to the parameters $a=\mi/2$ and $\tilde{c}_1=1-\mi$.}
\label{fig:=dp3+id2u0+1v0-1l50Reu}
\end{center}
\end{figure}
\begin{figure}[htpb]
\begin{center}
\includegraphics[height=50mm,width=100mm]{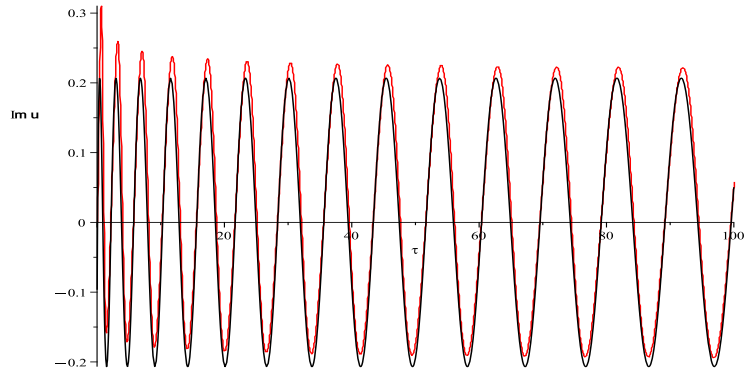}
\caption{The red and black plots are, respectively, the imaginary parts of the numerical and large-$\tau$ asymptotic values
of the function $u(\tau)$ for $\tau\in(0,100)$ corresponding to the parameters $a=\mi/2$ and $\tilde{c}_1=1-\mi$.}
\label{fig:dp3+id2u0+1v0-1l100Imu}
\end{center}
\end{figure}

In the second example, we also consider the case $a=\mi/2$; however, here, $\tilde{c}_1=-2+\mi6$, so that the parameter
$\nu+1=0.58885657\ldots+\mi0.13705579\ldots$, which means that we have to use the asymptotic formula from Theorem C.2 of
\cite{KitVar2023} since the asymptotics given in Theorem C.1 is not applicable. The results of the corresponding calculations
are presented in Figs.~\ref{fig:=dp3+id2u0-2v0+6l50Reu} and \ref{fig:dp3+id2u0-2v0+6l50Imu}. This example is especially
interesting from the point of view of visualization: it turns out that the first ``peak'' of the asymptotic plot
for $\mathrm{Re}\,u(\tau)$ extends from about $-13150$ to $2050$ along the $y$-axis, while the subsequent peaks have
more adequate heights; as clearly seen in Figure~\ref{fig:=dp3+id2u0-2v0+6l50Reu}, they are close to the heights of the
corresponding peaks of the numerical solution. The situation with the first imaginary ``up-down'' peak is even worse;
its height ranges from about $-60000$ to $62000$ along the $y$-axis.
As a result, the {\sc Maple} graphical program automatically scales the asymptotic
plots such that the first peaks fit within the frame of the figures, while the other peaks, which have normal sizes prior to
the scaling, resemble Lilliputians; then, when we plot the numerical and asymptotic graphs on one figure, we see that the
asymptotics are not correlated with the numerical solution, so that the initial visual impression indicates that the asymptotic
formula is wrong. When, however, the situation with the first peak, which is located on the interval (0.8,0.9), is clarified,
it becomes obvious that if we construct an asymptotic graph starting at $\tau=1$, then everything falls into place.
\begin{figure}[htpb]
\begin{center}
\includegraphics[height=50mm,width=100mm]{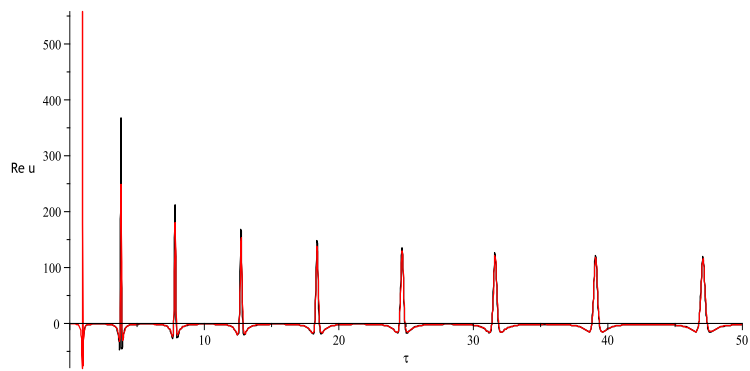}
\caption{The red plot is the real part of the numerical solution $u(\tau)$ corresponding to the parameters $a=\mi/2$ and
$\tilde{c}_1=-2+\mi6$ for $\tau\in(0,50)$. The black plot is the real part of the large-$\tau$ asymptotics of
$u(\tau)$ for $\tau\in(1,50)$.}
\label{fig:=dp3+id2u0-2v0+6l50Reu}
\end{center}
\end{figure}
\begin{figure}[htpb]
\begin{center}
\includegraphics[height=50mm,width=100mm]{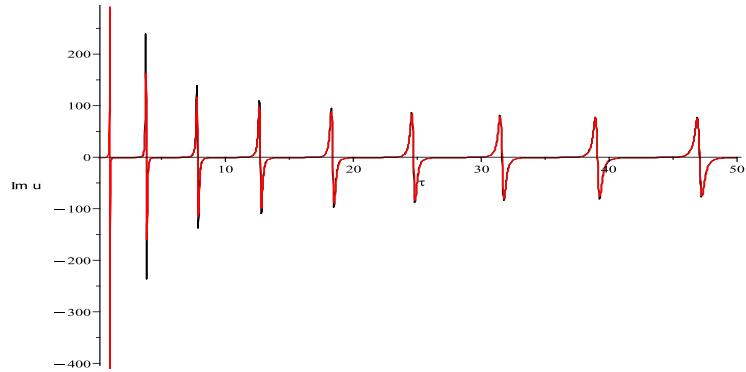}
\caption{The red plot is the imaginary part of the numerical solution $u(\tau)$ corresponding to the parameters $a=\mi/2$ and
$\tilde{c}_1=-2+\mi6$ for $\tau\in(0,50)$. The black plot is the imaginary part of the large-$\tau$ asymptotics of
$u(\tau)$ for $\tau\in(1,50)$.}
\label{fig:dp3+id2u0-2v0+6l50Imu}
\end{center}
\end{figure}

Finally, in the third example, we consider the case $a=-\mi/2$. Here, the parameter $\tilde{c}_1=-3-\mi2$, which implies that
$\nu+1=0.5454729\ldots+\mi0.0189800\ldots$. The value of $\mathrm{Re}\,\nu+1$ shows that, in order to find large-$\tau$
asymptotics of the corresponding solution $u(\tau)$, we have to apply Theorem C.2 of \cite{KitVar2023}; the results of these
calculations are presented in Figs.~\ref{fig:=dp3-id2u0-3v0-2l50Reu}--\ref{fig:dp3-id2u0-3v0-2l10Imu}. In this example,
we do not encounter any surprises; rather, we demonstrate a different type of behaviour for the function $u(\tau)$.
On the scale of Figs.~\ref{fig:=dp3-id2u0-3v0-2l50Reu} and \ref{fig:dp3-id2u0-3v0-2l50Imu}, the numerical and asymptotic plots
coincide almost exactly; as a result, we produced two close-up figures, namely, Figs.~\ref{fig:=dp3-id2u0-3v0-2l10Reu} and
\ref{fig:dp3-id2u0-3v0-2l10Imu}, so that the reader can evaluate the quality of the
approximation of the numerical solution by its large-$\tau$ asymptotics on small distances.
\begin{figure}[htpb]
\begin{center}
\includegraphics[height=50mm,width=100mm]{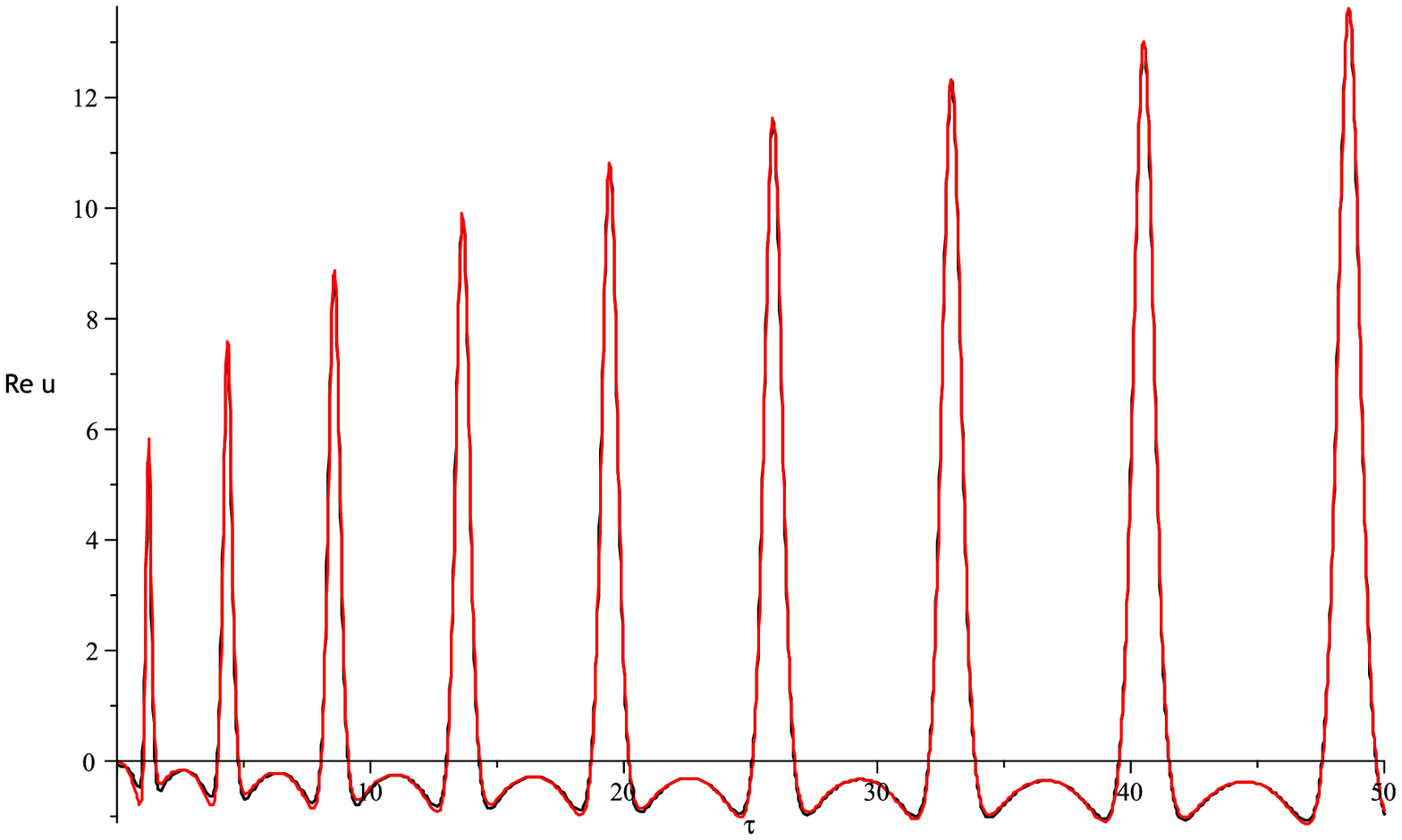}
\caption{The red and black plots are, respectively, the real parts of the numerical and large-$\tau$ asymptotic values
of the function $u(\tau)$ for $\tau\in(0,50)$ corresponding to the parameters $a=-\mi/2$ and $\tilde{c}_1=-3-\mi2$.}
\label{fig:=dp3-id2u0-3v0-2l50Reu}
\end{center}
\end{figure}
\begin{figure}[htpb]
\begin{center}
\includegraphics[height=50mm,width=100mm]{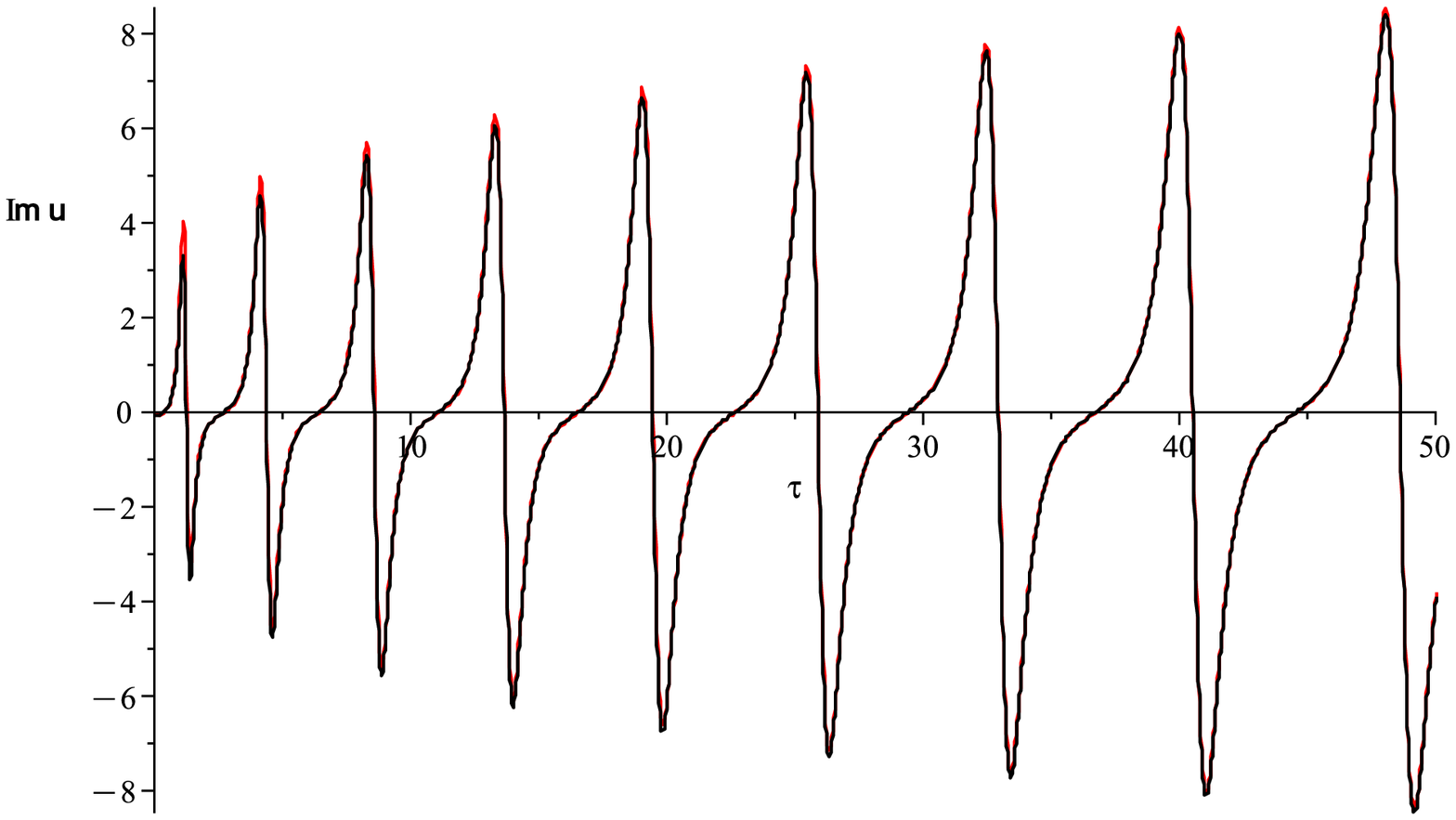}
\caption{The red and black plots are, respectively, the imaginary parts of the numerical and large-$\tau$ asymptotic values
of the function $u(\tau)$ for $\tau\in(0,50)$ corresponding to the parameters $a=-\mi/2$ and $\tilde{c}_1=-3-\mi2$.}
\label{fig:dp3-id2u0-3v0-2l50Imu}
\end{center}
\end{figure}
\begin{figure}[htpb]
\begin{center}
\includegraphics[height=50mm,width=100mm]{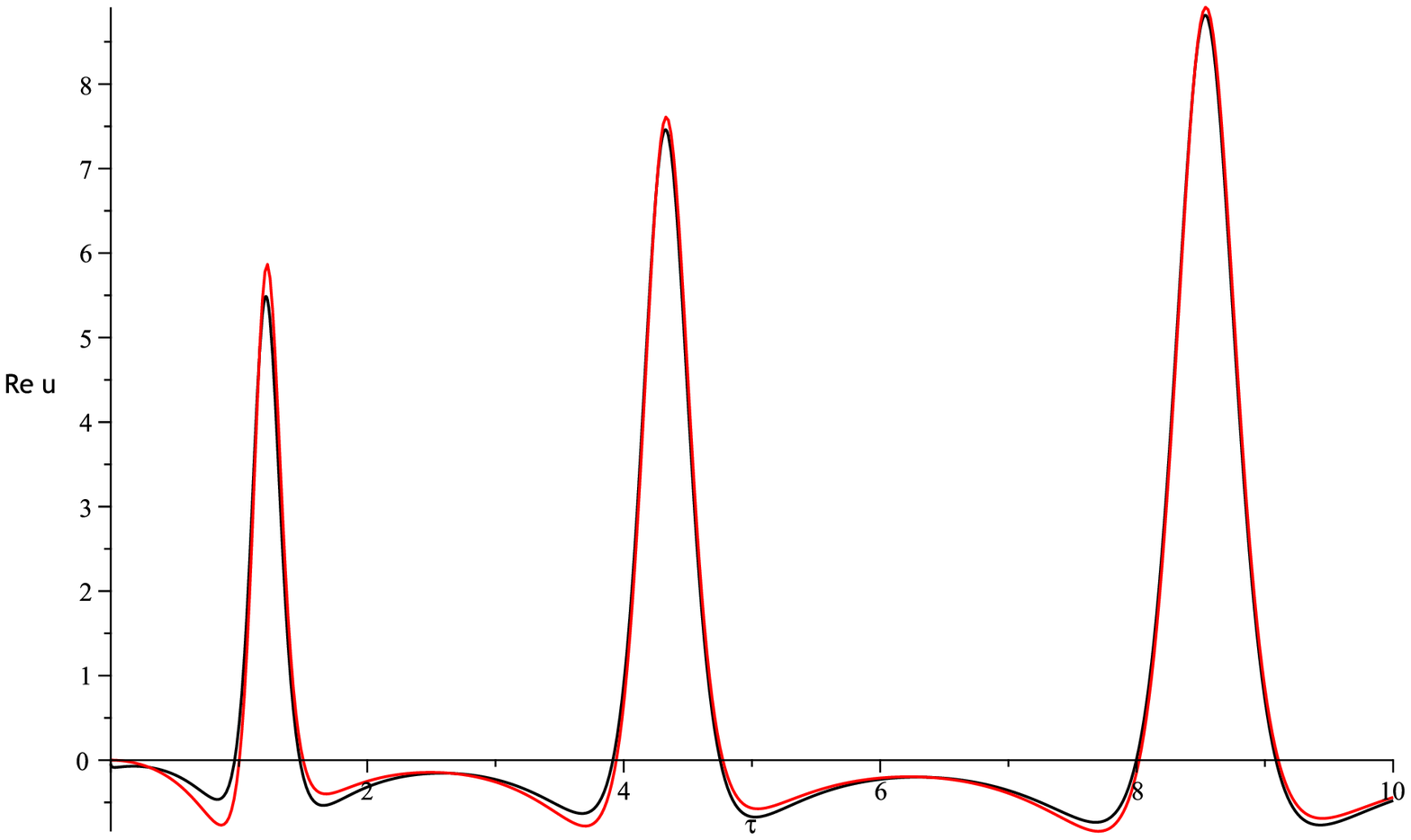}
\caption{The red and black plots are, respectively, the real parts of the numerical and large-$\tau$ asymptotic values
of the function $u(\tau)$ for $\tau\in(0,10)$ corresponding to the parameters $a=-\mi/2$ and $\tilde{c}_1=-3-\mi2$.}
\label{fig:=dp3-id2u0-3v0-2l10Reu}
\end{center}
\end{figure}
\begin{figure}[htpb]
\begin{center}
\includegraphics[height=50mm,width=100mm]{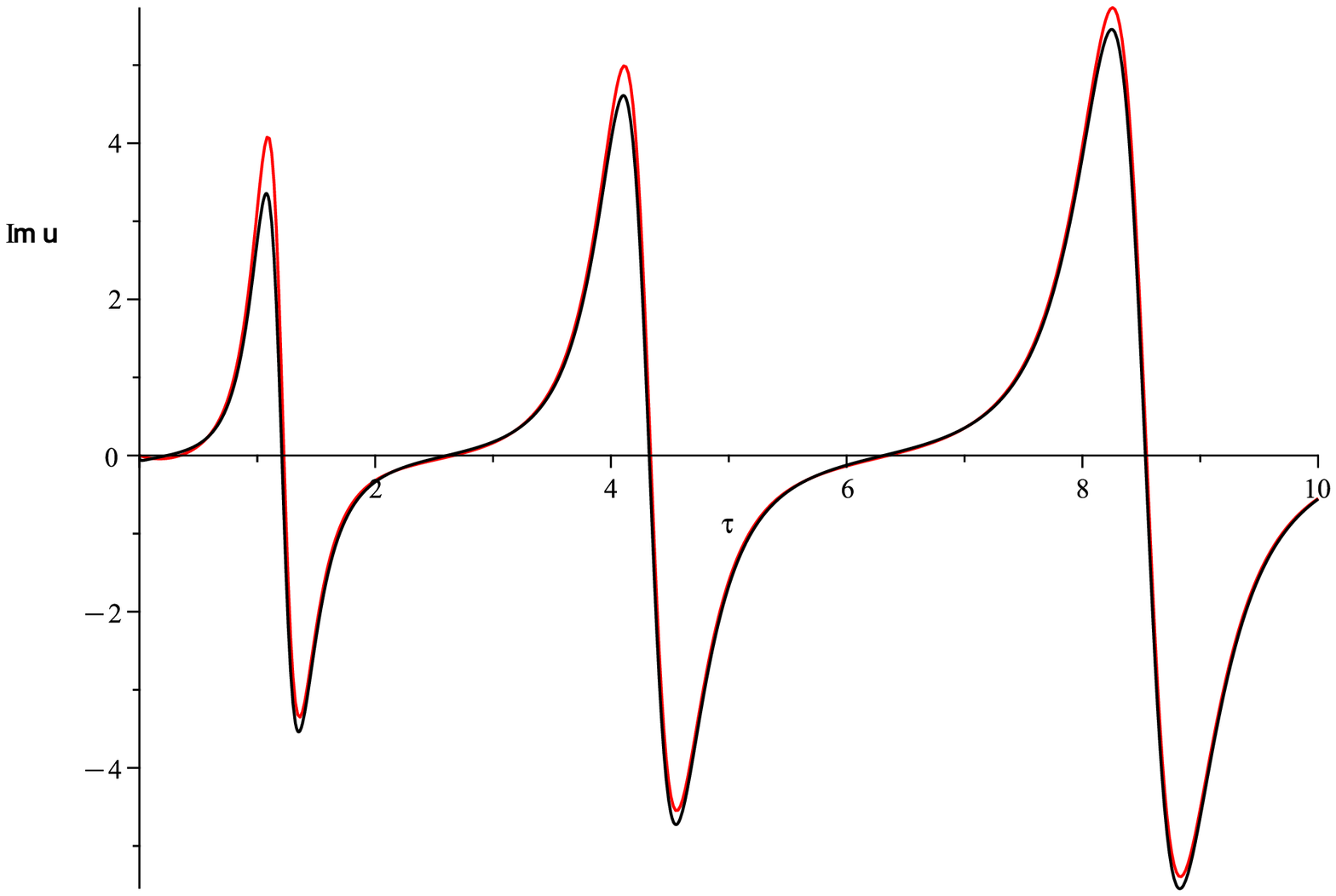}
\caption{The red and black plots are, respectively, the imaginary parts of the numerical and large-$\tau$ asymptotic values
of the function $u(\tau)$ for $\tau\in(0,10)$ corresponding to the parameters $a=-\mi/2$ and $\tilde{c}_1=-3-\mi2$.}
\label{fig:dp3-id2u0-3v0-2l10Imu}
\end{center}
\end{figure}
\pagebreak

\vspace*{0.35cm}
\noindent
\textbf{\Large Acknowledgements}
\vspace*{0.25cm}

\noindent
A.~V. is grateful to the St.~Petersburg Department of the Steklov Mathematical Institute (PDMI)
for hospitality during the preparation of this work.


\begin{thebibliography}{1000}
\addcontentsline{toc}{section}{References}
\bibitem{seq:A128135}
Adamson G.W.,
Sequence A128135 in The On-Line Encyclopedia of Integer Sequences (2007),
published electronically at https://oeis.org/
\bibitem{BE}
Erd\'{e}lyi A., Magnus W., Oberhettinger F., Tricomi F.~G.,
Higher transcendental functions, \textbf{1},
McGraw-Hill Book Company, Inc., New York-Toronto-London, 1953.
(Based, in part, on notes left by Harry Bateman.)
\bibitem{Eremin2019}
Eremin G.,
Legendre’s formula and p-adic analysis,
arXiv:1907.11902.
\bibitem{Garnier1912}
Garnier R.,
Sur des \'equations diff\'erentielles du troisi\`eme ordre dont l'int\'egrale g\'en\'erale est uniforme
et sur une classe d'\'equations nouvelles d'ordre sup\'erieur dont l'int\'egrale g\'en\'erale a ses points
critiques fixes,
{\it Annales scientifiques de l'\'Ecole Normale Sup\'erieure}, S\'erie 3: {\bf29}, pp. 1-126, 1912.
\bibitem{Gromak1973}
Gromak V. I.,
The solutions of Painlev\'e's third equation. (Russian)
\emph{Differencial'nye Uravnenija} \textbf{9:11} (1973), 2082--2083.
\bibitem{Ince}
Ince E. L.,
Ordinary Differential Equations,
Dover Publications, New York, 1944.
\bibitem{KitSIGMA2019}
Kitaev A. V.,
Meromorphic solution of the degenerate third Painlev\'{e}
equation vanishing at the origin,
\emph{SIGMA} \textbf{15} (2019), 046, 53pp.
\bibitem{KitVar2004}
Kitaev A. V., Vartanian A. H.,
Connection formulae for asymptotics of solutions of the degenerate third Painlev\'{e} equation: I,
\emph{Inverse Problems} \textbf{20} (2004), 1165--1206.
\bibitem{KitVar2010}
Kitaev A. V., Vartanian A.,
Connection formulae for asymptotics of solutions of the degenerate third Painlev\'{e} equation: II,
\emph{Inverse Problems} \textbf{26} (2010), 105010, 58pp.
\bibitem{KitVar2019}
Kitaev A. V., Vartanian A.,
Asymptotics of integrals of some functions related to the degenerate third Painlev\'{e} equation,
\emph{J. Math. Sci. (N.Y.)\/} \textbf{242} (2019), 715--721.
\bibitem{KitVar2023}
Kitaev A. V., Vartanian A.,
Alebroid solutions of the degenerate third Painlev\'e equation for vanishing formal monodromy parameter,
arXiv:2304.05671.
\bibitem{seq:A005187}
Sloane N. J. A.,
Wilks A.,
Sequence A005187 in The On-Line Encyclopedia of Integer Sequences (1991,1999),
published electronically at https://oeis.org/
\bibitem{seq:A004134}
Sloane N. J. A.,
Sequence A004134 in The On-Line Encyclopedia of Integer Sequences,
published electronically at https://oeis.org/
\bibitem{OEIS}
Sloane N. J. A.,
Sequences A085058, A001511 in
The On-Line Encyclopedia of Integer Sequences (2003),
published electronically at https://oeis.org.
\bibitem{Suleimanov2017}
Suleimanov B. I.,
Effect of a small dispersion on self-focusing in a
spatially one-dimensional case,
\emph{JETP Letters\/} \textbf{106} (2017), 400--405.
\end{thebibliography}
\end{document}